\documentclass[12pt]{article}
\usepackage{amsmath,amsthm,amssymb,amsfonts,amscd}
\usepackage[dvips]{graphics}
\usepackage[all]{xy}
\usepackage[numbers,sort&compress]{natbib}
\usepackage[mathscr]{eucal}
\usepackage{txfonts}

\makeatletter

\def\@maketitle{%
  \newpage
  \null
  \vskip 1em%
  \begin{center}%
  \let \footnote \thanks
    {\Large \@title \par}%
    \vskip 1.5em%
    {\large
      \lineskip .5em%
      \begin{tabular}[t]{c}%
        \@author
      \end{tabular}}%
  \end{center}%
  \par
  \vskip 1.5em}

\def\section{%
           \@startsection {section}{1}{\z@}%
           {3.5ex \@plus 1ex \@minus .2ex}%
           {1.5ex \@plus .2ex}%
           {\large\bfseries\boldmath}}
\def\subsection{%
           \@startsection{subsection}{2}{\z@}%
           {2.5ex \@plus 1ex \@minus .2ex}%
           {1ex \@plus .2ex}%
           {\normalfont\normalsize\itshape}}
\def\subsubsection{%
           \@startsection{subsubsection}{3}{\z@}%
           {2.5ex \@plus 1ex \@minus .2ex}%
           {1ex \@plus .2ex}%
           {\normalfont\normalsize\itshape}}
\def\paragraph{%
           \@startsection{paragraph}{4}{\z@}%
           {2ex \@plus 1ex \@minus .2ex}%
           {-1em}%
           {\normalfont\normalsize\itshape}}
\def\subparagraph{%
           \@startsection{subparagraph}{5}{\parindent}%
           {2ex \@plus 1ex \@minus .2ex}%
           {-1em}%
           {\normalfont\normalsize\itshape}}

\def\@seccntDot{.}
\def\@seccntformat#1{\csname the#1\endcsname\@seccntDot\hskip 0.5em}

\makeatother

\newenvironment{enm}%
{\begin{enumerate}
\setlength{\itemsep}{0em}
}%
{\end{enumerate}}
\newenvironment{itm}%
{\begin{itemize}
\setlength{\itemsep}{0em}
}%
{\end{itemize}}
{\begin{description}
\setlength{\itemsep}{0em}
}%
{\end{description}}

\numberwithin{equation}{section}

\newtheorem{theorem}{Theorem}[section]
\newtheorem{corollary}[theorem]{Corollary}
\newtheorem{proposition}[theorem]{Proposition}
\newtheorem{lemma}[theorem]{Lemma}

\theoremstyle{definition}    
\newtheorem{definition}[theorem]{Definition}
\newtheorem{example}[theorem]{Example}
\newtheorem{remark}[theorem]{Remark}
\newtheorem*{acknowledgements}{Acknowledgements}

\theoremstyle{remark}

\newcommand{\C}{\mathbb{C}} 
\newcommand{\CP}{\mathbb{P}} 
\newcommand{\Z}{\mathbb{Z}} 
\newcommand{\GL}{\operatorname{GL}} 
\newcommand{\gl}{\operatorname{\mathfrak{gl}}}
\newcommand{\g}{\mathfrak{g}} 
\newcommand{\ad}{\operatorname{ad}}
\newcommand{\End}{\operatorname{End}}
\newcommand{\Aut}{\operatorname{Aut}}
\newcommand{\Hom}{\operatorname{Hom}}
\newcommand{\Ker}{\operatorname{Ker}}
\newcommand{\range}{\operatorname{Im}}
\newcommand{\Coker}{\operatorname{Coker}}

\newcommand{\rank}{\operatorname{rank}}
\newcommand{\tr}{\operatorname{tr}}
\newcommand{\unit}{\mathrm{Id}}
\newcommand{\vin}{\operatorname{in}} 
\newcommand{\vout}{\operatorname{out}} 
\newcommand{\Res}{\operatorname*{Res}}
\newcommand{\lset}[2]%
{\left\{ \, \left. #1 \hspace{0.25em} \right| \, #2 \, \right\} }
\newcommand{\rset}[2]%
{\left\{ \, #1 \, \left| \hspace{0.25em} #2\right. \, \right\} } 
\newcommand{\ov}{\overline}
\newcommand{\bM}{\operatorname{\mathbf{M}}} 
\newcommand{\Mreg}{\operatorname{\mathfrak{M}}^\mathrm{irr}}
\newcommand{\mc}{\operatorname{\it mc}}
\newcommand{\EE}{\operatorname{\mathcal{E}}}
\newcommand{\vek}{{\vec{k}}}
\newcommand{\vel}{{\vec{l}}}
\newcommand{\Rep}{\operatorname{Rep}}

\newcommand{\bW}{\mathbf{W}}
\newcommand{\QQ}{\mathcal{Q}}
\newcommand{\OO}{\mathcal{O}}
\newcommand{\bO}{\mathbb{O}}
\newcommand{\st}{\text{\rm -st}}

\newcommand{\HD}{\operatorname{HD}}
\newcommand{\lDk}{{}_\vel\mathcal{D}_\vek}
\newcommand{\DDk}{\mathcal{D}_\vek}
\newcommand{\lHk}{{}_\vel\mathcal{H}_\vek}
\newcommand{\kDl}{{}_\vek\mathcal{D}_\vel}
\newcommand{\kHl}{{}_\vek\mathcal{H}_\vel}
\newcommand{\add}{\operatorname{\it add}}
\newcommand{\ord}{\operatorname{ord}}

\newcommand{\frh}{\mathfrak{h}}
\newcommand{\FF}{\operatorname{\mathcal{F}}}
\newcommand{\VV}{\mathcal{V}}
\newcommand{\WW}{\mathcal{W}}

\DeclareMathAlphabet{\mathFrak}{U}{euf}{m}{n}
\let\mathfrak\mathFrak

\usepackage{hyperref}
\hypersetup{%
 bookmarksnumbered=true,%
 pdftitle={Middle Convolution and Harnad Duality},%
 pdfauthor={Daisuke Yamakawa}%
}

\begin{document}

\title{Middle Convolution and Harnad Duality}
\author{Daisuke Yamakawa
\thanks{This research was supported by the grant 
ANR-08-BLAN-0317-01 
of the Agence nationale de la recherche.}
\thanks{current address: D\'epartement de math\'ematiques et applications, \'Ecole normale sup\'erieure, 45 rue d'Ulm, 75005 Paris, France;
\texttt{yamakawa@dma.ens.fr}} \\
\small\it
CMLS, \'Ecole polytechnique - CNRS UMR 7640 - \\
\small\it
ANR S\'EDIGA
}

\maketitle

\begin{abstract}
We interpret the additive middle convolution operation 
in terms of the Harnad duality, 
and as an application, generalize the operation 
to have a multi-parameter and act on irregular singular systems.
\end{abstract}

\section{Introduction}\label{sec:intro}

\noindent
The middle convolution 
introduced by Katz~\cite{Katz} and 
reformulated by V\"olklein~\cite{Volk}, 
Dettweiler-Reiter~\cite{DR-mc} is an operation acting on 
\begin{itm}
\item the category of local systems on a punctured projective line 
(in the multiplicative case); or
\item that of Fuchsian systems (in the additive case).
\end{itm}
The two multiplicative and additive operations match up via 
the Riemann-Hilbert correspondence~\cite{DR-rh}. 
Katz effectively used the middle convolution to study 
irreducible local systems which are {\em rigid}, namely,
have no deformation preserving the local monodromy data,
and proved that 
any such a local system 
is obtained by applying a finite iteration of 
tensor multiplications by rank 1 local systems 
and middle convolutions,
to some rank 1 local system~\cite{Katz}.
One can find many other applications of the middle convolution; 
in particular, 
to the Deligne-Simpson problem~\cite{Cra-par,CS,Kostov}, 
to the classification/connection problems~\cite{Oshima-conn}, 
and to the theory of isomonodromic deformations
~\cite{Boa-klein,DR-p,Filipuk,HF}.

In this article we focus attention on 
the additive middle convolution.
First recall its definition 
following Dettweiler-Reiter~\cite[Appendix]{DR-mc}.
Fix a finite set $D$ of points in $\C$ and 
suppose that a pair $(V,A)$ of 
a finite-dimensional $\C$-vector space $V$ 
and a Fuchsian system 
\begin{equation}\label{eq:fuchs}
\frac{du}{dz}=A(z)\,u, \quad 
A(z)=\sum_{t \in D} \frac{A_t}{z-t}, \quad A_t \in \End(V)
\end{equation}
with singularities on $D \cup \{ \infty \}$ is given.
Here we do not distinguish a system \eqref{eq:fuchs} and 
its coefficient matrix $A(z)$ fixing the coordinate $z$.
The definition of the middle convolution $\mc_\lambda(V,A)$ 
is divided into the following two steps.

\bigskip
\noindent
\textbf{(MC 1)}\;
Set $W_t := V/\Ker A_t,\, t \in D$, and let 
\begin{itm}
\item $Q_t \colon W_t \to V$ be the injection induced from $A_t$; and
\item $P_t \colon V \to W_t$ be the projection.
\end{itm}
Obviously we have $A_t =Q_t P_t$.
Set $W:=\bigoplus_{t \in D} W_t$ 
and let $Q \colon W \to V$ (resp.\ $P \colon V \to W$) 
be the linear map whose block components with respect to 
the decomposition $W=\bigoplus W_t$ are $Q_t$ (resp. $P_t$).

\medskip
\noindent
\textbf{(MC 2)}\;
For $\lambda \in \C$, 
set $V^\lambda:= W/\Ker (PQ +\lambda\,\unit_W)$, and let 
\begin{itm}
\item $Q^\lambda \colon W \to V^\lambda$ be 
the projection; and
\item $P^\lambda \colon V^\lambda \to W$ 
be the injection induced from $PQ +\lambda\,\unit_W$.
\end{itm}
Obviously we have $P^\lambda Q^\lambda = PQ +\lambda\,\unit_W$.
Let $Q^\lambda_t \colon W_t \to V^\lambda$ 
(resp.\ $P^\lambda_t \colon V^\lambda \to W_t$) 
be the block components of $Q^\lambda$ (resp.\ $P^\lambda$).
\medskip

\begin{definition}[Dettweiler-Reiter]
We call 
\[
\mc_\lambda (V,A) := (V^\lambda,A^\lambda), \quad
A^\lambda(z)= \sum_{t\in D} \frac{Q^\lambda_t P^\lambda_t}{z-t}
\]
the {\em (additive) middle convolution} of $(V,A)$ with $\lambda$.
\end{definition}

Here looking at the above procedure, 
one can observe that 
the given pair $(V,A)$ and its middle convolution are 
described as
\[
A(z)= Q(z\,\unit_W -T)^{-1}P, \quad 
A^\lambda(z)= Q^\lambda (z\,\unit_W -T)^{-1} P^\lambda,
\]
where $T := \bigoplus_t t\,\unit_{W_t} \in \End(W)$. 
Such an expression of a system can be found in the papers of 
Adams, Harnad, Hurtubise and Previato
~\cite{AHH-dual,AHH-iso2,AHP-iso1,Harnad} 
and implicitly in that of Jimbo-Miwa-M\^ori-Sato~\cite{JMMS}. 
In particular,
Harnad~\cite{Harnad} 
considered two systems having the following symmetric description:
\[
A(z) = S + Q(z\,\unit_W -T)^{-1}P, \quad 
B(\zeta)=T + P(\zeta\,\unit_V -S)^{-1}Q,
\]
where $Q \in \Hom(W,V),\, P \in \Hom(V,W)$, and 
$S,T$ are semisimple endomorphisms of $V,W$ respectively.
He then obtained an equivalence 
(called the {\em Harnad duality}) between 
the isomonodromic deformations of the systems $A(z)$ and $-B(\zeta)$.
Note that if $S=0$, we have $B(\zeta)=T+PQ\zeta^{-1}$.
So the procedure getting the middle convolution 
can be rephrased roughly as follows: for given Fuchsian system 
$A(z)=Q(z\,\unit_W -T)^{-1}P$,
\begin{enm}
\item[(a)] take its `Harnad dual'
\footnote{For convenience, we use the terminology `Harnad dual' 
on the system $B(\zeta)$, not on $-B(\zeta)$,
while the original `Harnad duality' is 
the correspondence between $A(z)$ and $-B(\zeta)$.} 
$B(\zeta)=T+PQ\zeta^{-1}$;
\item[(b)] shift $B(\zeta)$ by $\lambda \zeta^{-1}$; and then
\item[(c)] take the Harnad dual again.
\end{enm}
Such a relation between the middle convolution 
and the Harnad duality is already known by Boalch~\cite{Boa-klein,Boa-diff}.
It may be viewed as 
another formulation of Katz's interpretation of 
the middle convolution via Fourier transform~\cite[\S 2.10]{Katz}. 
Note that the above procedure makes sense 
even in the case that $S$ is an arbitrary semisimple endomorphism. 
Suppose that a system $A(z)=S+\sum A_t/(z-t)$ with 
simple poles on $D$ and a pole of order 2 at $\infty$ is given. 
Then at step (a), take its Harnad dual 
$B(\zeta)=T+P(\zeta\,\unit_V-S)^{-1}Q$.
Next at step (b), shift it by 
some rank 1 Fuchsian system $\alpha(\zeta)$ having 
singularities at the eigenvalues of $S$.
Finally at step (c), take the Harnad dual again.
Then we get the middle convolution $\mc_\alpha(A)$ 
with $\alpha$. 
Boalch generalized the Harnad duality, 
called the {\em cycling}, and 
obtained a further generalization of the middle convolution 
(see \cite[\S 4.6]{Boa-quiver}) for systems 
with simple poles at $D$ and a pole of order 3 at $\infty$  
which has a `normal form' (see Definition~\ref{dfn:HTL}).

If $T$ is not semisimple, then 
the matrix-valued function $Q(z\,\unit_W -T)^{-1}P$ 
has in general higher order poles 
at the eigenvalues of $T$. 
In fact, it is known~\cite{Kawakami,Wood} that 
for any system of the form
\begin{equation}\label{eq:irreg1}
A(z)=\sum_{t \in D}\sum_{k=1}^{k_t} \frac{A_{t,k}}{(z-t)^k},
\quad A_{t,k} \in \End(V), \quad k_t \in \Z_{>0},
\end{equation}
there exist a finite-dimensional $\C$-vector space $W$, 
an endomorphism $T$ of $W$ and 
homomorphisms $Q \colon W \to V$ and $P \colon V \to W$, such that
$A(z)=Q(z\,\unit_W-T)^{-1}P$. 
One may then expect that the middle convolution operation $\mc_\alpha$ 
can be generalized to that acting on systems of the form \eqref{eq:irreg1}.
In order to obtain such a generalization, we have to make 
a rigorous meaning of the `Harnad dual', because 
for given system $A(z)$, the choice of datum $(W,T,Q,P)$ satisfying 
$A(z)=Q(z\,\unit_W -T)^{-1}P$ is not unique, 
so we have to eliminate ambiguity of the choice 
in a certain canonical way.
In Fuchsian case, what we do in (MC~1) gives the answer to it,
so the problem is easily solved.
In this article, 
as a generalization of the procedure (MC~1), 
we give an explicit construction of $(W,T,Q,P)$ 
for any system $A(z)$ of the form \eqref{eq:irreg1}, 
and show that it is `canonical' in the following sense:
the constructed datum $(W,T,Q,P)$ together with $V$,
which we call the {\em canonical datum for $A(z)$} 
(Definition~\ref{dfn:canonical}), 
satisfies a stability condition (Definition~\ref{dfn:stable}) 
in the sense of Mumford's geometric invariant theory, 
and is characterized up to isomorphism via this condition.
More precisely, we show the following:

\begin{theorem}[Proposition~\ref{prop:unique} and Proposition~\ref{prop:explicit}] 
For any system $A(z)$ of the form as in \eqref{eq:irreg1}
with $V \neq 0$, the canonical datum is stable;
in particular, 
there exists a stable datum $(V,W,T,Q,P)$ satisfying
$Q(z\,\unit_W-T)^{-1}P=A(z)$.

If two data $(V,W,T,Q,P)$ and $(V,W',T',Q',P')$ with the same $V \neq 0$ 
are both stable and satisfy
\[
Q(z\,\unit_W-T)^{-1}P=Q'(z\,\unit_{W'}-T')^{-1}P',
\]
then there exists an isomorphism $f \colon W \to W'$ such that
\[
Q' = Q f^{-1}, \quad P' = f P, \quad T' = f T f^{-1}.
\]
\end{theorem}

Using the canonical data, we can define the notion of Harnad dual 
as follows: for given system of the form 
\begin{equation}\label{eq:irreg2}
A(z)=S + A^0(z), \quad 
A^0(z)=\sum_{t \in D}\sum_{k=1}^{k_t} \frac{A_{t,k}}{(z-t)^k},
\quad S,\, A_{t,k} \in \End(V), \quad k_t \in \Z_{>0},
\end{equation}
take the canonical datum $(V,W,T,Q,P)$ for 
the system $A^0(z)$, and set 
$B(\zeta):=T+P(\zeta\,\unit_V -S)^{-1}Q$.
We call the pair $(W,B)$ as the Harnad dual of $(V,A)$ 
and denote it by $\HD(V,A)$ (Definition~\ref{dfn:dual}).
Then the above characterization in terms of the 
geometric invariant theory gives 
the following two basic properties of $\HD$:

\begin{theorem}[Theorem~\ref{thm:duality-cat}]\label{thm:intro-dualcat}
Suppose that a pair $(V,A)$ of the form as in \eqref{eq:irreg2} with $V \neq 0$
satisfies the following two conditions:
\begin{enm}
\item[{\rm (a)}] $(V,A)$ is irreducible, 
namely, $V$ has no nonzero proper subspace preserved by 
all $A_{t,k}$ and $S$;
\item[{\rm (b)}] $(V,A)$ is not equivalent to 
a pair of the form $(\C,s),\, s \in \C$, 
under constant gauge transformation.
\end{enm}
Then $\HD(V,A)$ is also irreducible and 
$\HD \circ \HD(V,A)$ is equivalent to $(V,A)$
under constant gauge transformation.
\end{theorem}

\begin{theorem}[see Theorem~\ref{thm:duality-geom} for the rigorous statement]
The correspondence between $(V,A)$ and $(W,-B)$, 
where $(W,B)=\HD(V,A)$, 
gives a symplectomorphism 
between naive moduli spaces of irreducible systems 
of the form as in \eqref{eq:irreg2} having different 
singularities with different truncated formal types.
\end{theorem}

Now the middle convolution (Definition~\ref{dfn:extmc}) is defined as 
\[
\mc_\alpha(V,A) := \HD \circ \add_\alpha \circ \HD(V,A),
\]
where the parameter $\alpha(\zeta)$ is a rank 1 system having singularities 
at the eigenvalues of $S=\lim_{z \to \infty} A(z)$, 
and $\add_\alpha \colon (W,B) \mapsto (W,B+\alpha)$ 
is the addition operator. 
Note that here we do not require that $S,T$ are semisimple 
or $\alpha$ is Fuchsian.
We obtain the following properties of 
$\mc_\alpha$, which are well-known in the original case,
as corollaries of the above two results on $\HD$:

\begin{corollary}[Corollary~\ref{cor:mc-property1} and Corollary~\ref{cor:mc-property2}]
Suppose that a pair $(V,A)$ of the form as in \eqref{eq:irreg2} 
satisfies Conditions~{\rm (a)} and {\rm (b)} 
in Theorem~\ref{thm:intro-dualcat}.
Then $\mc_0(V,A) \sim (V,A)$.
Furthermore, if a rank 1 system $\alpha(\zeta)$ with singularities at the eigenvalues of $S$ satisfies $\mc_\alpha(V,A) \neq (0,0)$,
then $\mc_\alpha(V,A)$ is also irreducible and
\[
\mc_\beta \circ \mc_\alpha (V,A) 
\sim \mc_{\alpha +\beta}(V,A)
\]
for any $\beta$.
\end{corollary}

\begin{corollary}[see Corollary~\ref{cor:mc-isom} for the rigorous statement]
The middle convolution $\mc_\alpha$ gives a symplectomorphism 
between naive moduli spaces of irreducible systems 
of the form as in \eqref{eq:irreg2} having different 
singularities with different truncated formal types.
\end{corollary}

Arinkin also generalized Katz's middle convolution 
to the irregular singular case 
in $\mathscr{D}$-module setting~\cite{Arin-mc},
and generalized Katz's algorithm 
by adding the Fourier transform to the original one
~\cite{Arin-rigid}.
We show using our generalized middle convolution 
that Katz's algorithm works well for 
`naively rigid systems having a normal form'
(see Definition~\ref{dfn:rigid} and Definition~\ref{dfn:HTL}), 
which corresponds to the case that 
the Fourier transform is not necessary 
in Arinkin's generalized algorithm:

\begin{theorem}[Corollary~\ref{cor:Katz}]
Suppose that a pair $(V,A)$ of the form as in \eqref{eq:irreg1} 
with $\dim V \geq 2,\, \Res_{z=\infty} A(z)=0$ 
is irreducible, naively rigid, 
and has a normal form at any $t \in D$.
Then applying a suitable finite iteration of operations of the form
\[
\add_\alpha \circ \mc_{\lambda/\zeta} \circ \add_\alpha, \quad
\alpha \in \EE_\vek(\C),\; \lambda =\Res_{z=\infty} \alpha(z),
\]
makes $(V,A)$ into an irreducible pair of rank 1.
\end{theorem}

\section{Setting}\label{sec:setting}

For a positive integer $k$, we set
\[
S_k := \C[z]/z^k\C[z], \quad S^k := z^{-k}\C[z]/\C[z].
\]
$S_k$ is a $k$-dimensional $\C$-algebra, 
and the pairing 
\[
S_k \otimes_\C S^k \to \C; \quad (f,g) \mapsto \Res_{z=0} (f(z)g(z))
\]
gives an identification $S_k^*=\Hom_\C (S_k,\C) \simeq S^k$.

For a finite-dimensional $\C$-vector space $V$, set 
\[
G_k(V):=\Aut_{S_k}(S_k \otimes_\C V), \quad
\g_k(V):=\End_{S_k}(S_k \otimes_\C V).
\]
Any element of $\g_k(V)$ is uniquely written as 
\[
\xi_0 + \xi_1 z + \cdots + \xi_{k-1} z^{k-1}, 
\quad \xi_i \in \End(V),
\]
and $G_k(V)$ is just the subset of $\g_k(V)$ 
defined by the condition $\det \xi_0 \neq 0$.
It has naturally a structure of complex algebraic group 
and the associated Lie algebra is nothing but $\g_k(V)$. 
The $\C$-dual $\g_k^*(V)$ of $\g_k(V)$ can be identified with the set
\[
\{\, \eta_1 z^{-1} + \eta_2 z^{-2} + \cdots + \eta_k z^{-k} 
\mid \eta_i \in \End(V) \,\}.
\]
For $\eta=\sum_j \eta_j z^{-j} \in \g^*_k(V)$, 
we denote by $\ord(\eta)$ the pole order of $\eta$;
\begin{equation}\label{eq:order}
\ord(\eta):=\max(\{\,j \geq 1 \mid \eta_j \neq 0\,\} \cup \{ 0 \}),
\end{equation}
which is preserved under the $G_k(V)$-coadjoint action.

For a collection $\vek=(k_t)_{t \in D}$ of positive integers 
indexed by a finite set $D$ of points in $\C$, 
we set
\[
\EE_\vek (V) := 
\lset{A(z)=\sum_{t \in D}\sum_{k=1}^{k_t} \frac{A_{t,k}}{(z-t)^k}}%
{A_{t,k} \in \End(V)},
\]
and regard $A(z) \in \EE_\vek (V)$ as 
a system of linear ordinary differential equations
\[
\frac{du}{dz} = A(z)u
\]
with singularities on $D \cup \{ \infty \}$.
If we define
\[
G_\vek (V) := \prod_{t \in D} G_{k_t}(V), \quad 
\g_\vek (V) := \bigoplus_{t \in D} \g_{k_t}(V), \quad
\g_\vek^* (V) := \bigoplus_{t \in D} \g_{k_t}^*(V),
\]
then $\EE_\vek (V)$ is isomorphic to $\g_\vek^*(V)$ by the map
\[
\EE_\vek (V) \xrightarrow{\simeq} \g_\vek^*(V); \quad 
A(z) \mapsto \left( \sum A_{t,k} z^{-k} \right)_{t \in D},
\]
and hence $G_\vek (V)$ naturally acts on $\EE_\vek (V)$.

Now for two finite-dimensional $\C$-vector spaces $V$ and $W$, we set
\[
\bM(V,W) := \Hom(W,V) \oplus \Hom(V,W).
\]
It has a natural symplectic structure 
\[
\omega = \tr dQ \wedge dP, \quad (Q,P) \in \bM(V,W),
\]
and the group $\GL(V) \times \GL(W)$ 
acts symplectomorphically on $\bM(V,W)$ by
\[
(a,b) \cdot (Q,P) := (aQb^{-1}, bPa^{-1}), \quad 
(a,b) \in \GL(V) \times \GL(W).
\]
Note that the map $(\mu_V,\mu_W) \colon \bM(V,W) \to \gl(V) \oplus \gl(W)$ defined by
\[
\mu_V(Q,P) := QP, \quad \mu_W (Q,P) := -PQ,
\]
is a moment map generating the $\GL(V) \times \GL(W)$-action, 
where we identify the Lie algebras $\gl(V), \gl(W)$ 
with those duals via the trace pairing. 
Let $T \in \End(W)$ be an endomorphism with eigenvalues in $D$ and
$W=\bigoplus_{t \in D} W_t$ be its generalized eigenspace decomposition. 
Let $N_t := T|_{W_t}-t\,\unit_{W_t} \in \End(W_t)$ be the
nilpotent part of $T$ restricted to $W_t$. 
Then we can consider the map
\[
\Phi_T \colon \bM(V,W) \to \EE_\vek(V); \quad 
(Q,P) \mapsto Q(z\,\unit_W - T)^{-1}P,
\]
where $k_t= \min \{ j \in \Z_{>0}; N_t^j=0 \}$.
To see it is well-defined, for $(Q,P) \in \bM(V,W)$, 
let $(Q_t,P_t) \in \bM(V,W_t)$ be its $\bM(V,W_t)$-component 
with respect to the decomposition $\bM(V,W)=\bigoplus \bM(V,W_t)$ 
induced from $W=\bigoplus W_t$.
Then we have 
\begin{align*}
Q \left( z\,\unit_W - T \right)^{-1} P &=
\sum_{t \in D} Q_t \left[ (z-t)\,\unit_{W_t} -N_t \right]^{-1} P_t \\
&= \sum_{t \in D} (z-t)^{-1} 
   Q_t \left[ \unit_{W_t} -(z-t)^{-1}N_t \right]^{-1} P_t \\
&= \sum_{t \in D} \sum_{k=1}^{k_t} \frac{Q_t N_t^{k-1} P_t}{(z-t)^k}.
\end{align*}
Therefore $\Phi_T(Q,P)$ can be considered as an element of $\EE_\vek (V)$. 

\begin{theorem}[Adams-Harnad-Hurtubise-Previato]\label{thm:AHHP}
For $g=(g_t(z)) \in G_\vek (V)$ 
and $(Q,P) \in \bM(V,W)$, let  
$(g \cdot Q,g \cdot P) \in \bM(V,W)$ 
be the point given by the following formulae:
\begin{align*}
g_t(z)Q_t (z-N_t)^{-1} &= (g \cdot Q)_t(z-N_t)^{-1} 
+ \text{\rm holomorphic}, \\ 
(z-N_t)^{-1}P_t g_t^{-1}(z) &= (z-N_t)^{-1}(g \cdot P)_t 
+ \text{\rm holomorphic}.
\end{align*}
Then it gives a well-defined Hamiltonian action of 
$G_\vek (V)$ on $\bM(V,W)$ 
with moment map $\Phi_T$.
\end{theorem}

More explicitly, $(g \cdot Q,g \cdot P)$ is defined by
\begin{equation}\label{eq:action}
(g\cdot Q)_t \equiv g_t \cdot Q_t 
:= \sum_{k=0}^{k_t-1} g_{t,k} Q_t N_t^k, 
\quad 
(g\cdot P)_t \equiv g_t \cdot P_t 
:= \sum_{k=0}^{k_t-1} N_t^k P_t (g_t^{-1})_k, 
\end{equation}
where $g_t(z)=\sum g_{t,k} z^k,\; g_t^{-1}(z)=\sum (g_t^{-1})_k z^k$. 
Note that the above gives 
actions on $\Hom(W,V)$ and $\Hom(V,W)$ separately 
and the two are coadjoint to each other. 

The proof was given by Adams-Harnad-Previato~\cite{AHP-iso1} 
in the case that $T$ is semisimple 
and by Adams-Harnad-Hurtubise~\cite{AHH-dual} in general cases.
Strictly speaking, the original result is stated in terms of loop group action, 
however one can easily derive the above from it.

Recall that any moment map is Poisson. Therefore the above theorem connects 
two Poisson manifolds $\bM(V,W)$ and $\EE_\vek (V)$, 
the space of systems of linear differential equations, via $\Phi_T$.

Now, take one more finite set $E$ of points in $\C$ 
and a collection $\vel =(l_s)_{s\in E}$ of positive integers 
indexed by $E$.
Let $S$ be an endomorphism of $V$ with eigenvalues in $E$ 
and $V=\bigoplus_{s \in E} V_s$ be its generalized eigenspace 
decomposition. 
For any $s \in E$, set $M_s := S|_{V_s}-s\,\unit_{V_s}$ and 
suppose $M_s^{l_s}=0$.
Then we can also define
\[
\Psi_S \colon \bM(V,W) \to \EE_\vel(W), \quad
(Q,P) \mapsto -P(\zeta\,\unit_V -S)^{-1}Q,
\]
where we denote the indeterminate by $\zeta$ instead of $z$.
It is a moment map generating the following $G_\vel(W)$-action 
on $\bM(V,W)$:
\[
(h\cdot Q)_s \equiv h_s \cdot Q_s 
:= \sum_{l=0}^{l_s-1} M_s^l Q_s (h_s^{-1})_l, \quad
(h\cdot P)_s \equiv h_s \cdot P_s 
:= \sum_{l=0}^{l_s-1} h_{s,l} P_s M_s^l, 
\] 
where $(Q_s,P_s)$ is the $\bM(V_s,W)$-component of $(Q,P)$ and
\[
h=(h_s)_{s \in E} \in G_\vel(W), \quad 
h_s(\zeta)=\sum_{l=0}^{l_s} h_{s,l} \zeta^l,\; 
h_s^{-1}(\zeta)=\sum_{l=0}^{l_s} (h_s^{-1})_l \zeta^l.
\]

Throughout this article, we fix two nonempty finite sets 
$D, E$ of points in $\C$, 
and collections $\vek =(k_t)_{t \in D},\, \vel =(l_s)_{s \in E}$ of 
positive integers indexed by $D, E$ respectively.

\section{Preliminary results from GIT}\label{sec:quiver}

Recall that a {\em quiver}
is a quadruple $\QQ=(I,\Omega,\vout,\vin)$ 
consisting of two sets $I, \Omega$ 
and two maps $\vout, \vin \colon \Omega \to I$.
The sets $I,\,\Omega$ are called 
the {\em set of vertices}, the {\em set of arrows}
respectively, 
and each $h \in \Omega$ is viewed as an arrow 
drawn from the vertex $\vout(h)$ to the vertex $\vin(h)$.

A {\em representation} of the quiver $\QQ$ 
is a pair consisting of a collection of $\C$-vector spaces $V_i$ 
indexed by $i \in I$, 
and a collection of linear maps 
$x_h \colon V_{\vout(h)} \to V_{\vin(h)}$ indexed by $h \in \Omega$.
For two representations given by 
$V'_i, x'_h$ and $V_i, x_h$, 
a {\em morphism} from the former to the latter 
is a collection of linear maps $\psi_i \colon V'_i \to V_i$ 
indexed by $i \in I$, such that 
$\psi_{\vin(h)} \circ x'_h = x_h \circ \psi_{\vout(h)}$ 
for all $h \in \Omega$.
It is called an {\em isomorphism} if each $\psi_i$ is an isomorphism. 
If each $\psi_i$ is just an injection, 
the collection $(\range \psi_i)_{i \in I}$ 
together with the linear maps 
$\range \psi_{\vout(h)} \xrightarrow{x_h} \range \psi_{\vin(h)}$ 
induced from $x_h$ 
gives a representation of $\QQ$, 
which is called a {\em subrepresentation} of $(V_i, x_h)$.
Note that subspaces $X_i \subset V_i,\, i \in I$ give a subrepresentation 
of $(V_i,x_h)$ if and only if $x_h(X_{\vout(h)}) \subset X_{\vin(h)}$ 
for all $h \in \Omega$.

For a positive integer $r \in \Z_{>0}$, 
let $\QQ_r$ be the quiver 
with set of vertices 
$D \cup \{ \infty \}$ obtained by 
drawing $r$ arrows both from $t$ to $\infty$ and 
$\infty$ to $t$, 
and an edge-loop (i.e., an arrow $h$ with $\vin(h) =\vout(h)$) at $t$ 
for each $t \in D$.
Let $\QQ \equiv \QQ_1$ for simplicity.

Each quintuple $(V,W,T,Q,P)$ consisting of 
\begin{itm}
\item two finite-dimensional $\C$-vector spaces $V, W$;
\item an endomorphism $T$ of $W$ whose eigenvalues are all contained in $D$; and 
\item a point $(Q,P) \in \bM(V,W)$
\end{itm}
gives a representation of $\QQ$ 
in the following way:
\begin{enm}
\item[(a)] for each $t \in D$, 
assign the vertex $t$ with the vector space 
$W_t:=\Ker (T-t\,\unit_W)^{\dim W}$, 
and assign the vertex $\infty$ with $V$;
\item[(b)] assign the arrow from $t$ to $\infty$ 
with $Q_t$, the $\Hom(W_t,V)$-block component of $Q$; 
and similarly, 
\item[(c)] assign the arrow from $\infty$ to $t$ 
with $P_t \in \Hom(V,W_t)$;
\item[(d)] assign the loop at $t$ with $N_t:=(T-t\,\unit_W) |_{W_t}$.
\end{enm}
We call such a quintuple $(V,W,T,Q,P)$ satisfying further
the following condition just as a {\em datum}:
\begin{equation}\label{eq:T}
\prod_{t \in D} (T-t\,\unit_W)^{k_t} =0,
\end{equation}
where the zero datum $(0,0,0,0,0)$, which corresponds to 
the zero representation of $\QQ$, is understood to satisfy the above.

\begin{definition}
A {\em subrepresentation} of a datum $(V,W,T,Q,P)$ is a 
pair $(X,Y)$ of vector subspaces $X \subset V,\, Y \subset W$ 
satisfying that 
\[
P(X) \subset Y, \quad Q(Y) \subset X, \quad T(Y) \subset Y.
\]
\end{definition}

If $(X,Y)$ is a subrepresentation of a datum $(V,W,T,Q,P)$, 
then we can consider the restriction $(X,Y,T|_Y,Q|_Y,P|_X)$ 
of the datum to the subspaces $X$ and $Y$, 
which obviously gives a subrepresentation of the representation 
of $\QQ$ corresponding to $(V,W,T,Q,P)$, and vice versa.

Representations of $\QQ$ 
with prescribed vector spaces $V$ at $\infty$ and 
$W_t$ at $t \in D$ form a vector space
\[
\Rep_\QQ(V,\bW) := 
\bigoplus_{t \in D} \Bigl( \bM(V,W_t) \oplus \End(W_t) \Bigr), \quad 
\bW=(W_t)_{t \in D},
\]
and the group 
$\GL(V) \times \GL(\bW)$, where $\GL(\bW) := \prod_{t \in D} \GL(W_t)$, 
acts on it as isomorphisms of representations.

\begin{definition}
A nonzero datum $(V,W,T,Q,P)$ is said to be {\em irreducible} 
if it has no subrepresentations except $(X,Y)=(0,0),\, (V,W)$,
or equivalently, 
if the corresponding representation of $\QQ$ is irreducible.
\end{definition}

We also need the following condition.
\begin{definition}\label{dfn:stable}
A datum $(V,W,T,Q,P)$ with $V \neq 0$ is said to be {\em stable}
if for any subrepresentation $(X,Y)$ of it, 
the equalities $X=0, V$ imply $Y=0, W$ respectively.
\end{definition}

Note that 
under the assumption $V \neq 0$, 
the irreducibility condition implies the stability condition.


\begin{remark}\label{rem:framing}
Let us fix $V \neq 0$ and consider the quiver $\QQ_{\dim V}$. 
Representations of $\QQ_{\dim V}$ with prescribed vector spaces 
$\C$ at $\infty$ and $W_t$ at $t \in D$,
form a vector space
\[
\Rep_{\QQ_{\dim V}}(\C,\bW)=
\bigoplus_{t \in D} 
\Bigl( \bM(\C,W_t)^{\oplus \dim V} \oplus \End(W_t) \Bigr).
\]
On the other hand, fixing a basis of $V$, we have identifications
\[
\Hom(W_t,V) \simeq \Hom(W_t,\C)^{\oplus \dim V}, 
\quad
\Hom(V,W_t) \simeq \Hom(\C,W_t)^{\oplus \dim V}.
\]
Thus we have an isomorphism
\[
\Rep_{\QQ}(V,\bW) \simeq 
\Rep_{\QQ_{\dim V}}(\C,\bW),
\]
which enables us to regard any datum with fixed $V$ 
as a representation of the quiver $\QQ_{\dim V}$, 
of which the vector space at the vertex $\infty$ is just $\C$.

Now fix a datum $(V,W,T,Q,P)$ and 
suppose that a subrepresentation of 
the corresponding representation of $\QQ_{\dim V}$ is given. 
Then one of the following two cases occurs:
\begin{enm}
\item[(a)] the vector space of it at the vertex $\infty$ is $0$;
\item[(b)] the vector space of it at the vertex $\infty$ is $\C$.
\end{enm}
In the first case, the subrepresentation gives 
a vector subspace $Y_t \subset W_t$ for each $t \in D$ such that
\[
Q_t(Y_t) =0, \quad N_t(Y_t) \subset Y_t,
\]
namely, the pair $(0,Y),\, Y=\bigoplus_{t \in D} Y_t$ is a subrepresentation 
of $(V,W,T,Q,P)$. 
In the second case, the subrepresentation gives 
a vector subspace $Y_t \subset W_t$ for each $t \in D$ such that
\[
P_t(V) \subset Y_t, \quad N_t(Y_t) \subset Y_t,
\]
namely, the pair $(V,Y),\, Y=\bigoplus_{t \in D} Y_t$ 
is a subrepresentation 
of $(V,W,T,Q,P)$.

This observation shows that when $V \neq 0$,
a datum $(V,W,T,Q,P)$ is stable if and only if the 
corresponding representation of $\QQ_{\dim V}$ is irreducible.
\end{remark}

\begin{lemma}\label{lem:stable-free}
Let $(Q_t,P_t,N_t)_{t \in D} \in \Rep_{\QQ}(V,\bW)$ 
be the point corresponding to some datum $(V,W,T,Q,P)$. 
If $V \neq 0$ and $(V,W,T,Q,P)$ is stable, then
the stabilizer at the point
$(Q_t,P_t,N_t)_{t \in D} \in \Rep_{\QQ}(V,\bW)$ 
of $\GL(\bW)$ is trivial.
\end{lemma}

\begin{proof}
Suppose that 
$b=(b_t) \in \GL(\bW)$ stabilizes 
$(Q_t,P_t,N_t)_{t \in D}$, namely,
\[
Q_t b_t^{-1} = Q_t, \quad 
b_t P_t =P_t, \quad
b_t N_t b_t^{-1} = N_t.
\] 
Set $Y_t:=\Ker (b_t -1) \subset W_t$ and $Y := \bigoplus Y_t$.
Then the above implies that
\[
N_t (Y_t) \subset Y_t, 
\quad
\range P_t \subset Y_t.
\]
Thus $(V,Y)$ is a subrepresentation 
of $(V,W,T,Q,P)$. 
By the stability condition we must have $Y=W$ and hence $b=1$.
\end{proof}

\begin{lemma}\label{lem:polystable}
Suppose that $V \neq 0$ and
let $(Q_t,P_t,N_t)_{t \in D} \in \Rep_{\QQ}(V,\bW)$ 
be the point corresponding to some datum $(V,W,T,Q,P)$. 
Then the $\GL(\bW)$-orbit of it is closed 
if and only if there exists a direct sum decomposition
\[
W = W(0) \oplus W(1)^{\oplus m_1} \oplus \cdots \oplus W(N)^{\oplus m_N}
\]
compatible with the decomposition $W=\bigoplus_t W_t$ such that:
\begin{enm}
\item[{\rm (a)}] $Q(W(i)^{\oplus m_i}) =0$ for all $i \geq 1$ and $\range P \subset W(0)$;
\item[{\rm (b)}] $T$ preserves each direct summand;
\item[{\rm (c)}] for any $ i\geq 1$ there exists $T_i \in \End(W(i))$ such that
$T |_{W(i)^{\oplus m_i}} = \unit_{\C^{m_i}} \otimes T_i$; 
\item[{\rm (d)}] the datum $(V,W(0),T |_{W(0)}, Q,P)$ given by 
restricting $(V,W,T,Q,P)$ to $W(0)$ is stable;
\item[{\rm (e)}] the datum $(0,W(i),T_i, 0,0)$ is irreducible 
for any $i \geq 1$.
\end{enm}
Moreover such a direct decomposition is unique up to 
permutation on $\{ 1, \dots ,N \}$.
\end{lemma}

\begin{proof}
Let $x \in \Rep_{\QQ_{\dim V}}(\C,\bW)$ 
be the point corresponding to $(V,W,T,Q,P)$ 
under the isomorphism given in Remark~\ref{rem:framing}. 
It is well-known~\cite{LP} that 
the orbit $\GL(\bW) \cdot x$ is closed if and only if 
$x$ is semisimple as a representation of $\QQ_{\dim V}$, 
namely there exists a direct sum decomposition
\[
x = x(0)^{\oplus m_0} \oplus x(1)^{\oplus m_1} 
\oplus \cdots \oplus x(N)^{\oplus m_N}
\]
by irreducible representations $x(i)$, 
and furthermore such a decomposition is unique up to 
permutation on $\{ 0,1, \dots , N \}$.

Since the vector space at the vertex $\infty$ 
of the representation $x$ is just $\C$,
we may assume that the vector space at $\infty$ of $x(0)$ is nonzero 
and those of all the other $x(i)$ are zero. 
Then we must have $m_0=1$ and the vector space at $\infty$ of $x(0)$ 
is just $\C$.
Now for each $i \geq 0$, 
the vector space $W(i)$ is given by the direct sum of those of $x(i)$ 
among all the vertices contained in $D$.
\end{proof}

\begin{lemma}\label{lem:stability}
A datum $(V,W,T,Q,P)$ with $V \neq 0$ is stable 
if and only if the following two conditions hold for any $t \in D$:
\begin{enm}
\item[{\rm (a)}] $\Ker Q_t \cap \Ker N_t =0$;
\item[{\rm (b)}] $\range P_t + \range N_t =W_t$.
\end{enm}
\end{lemma}

\begin{proof}
Without loss of generality we may assume $D=\{ 0 \}$, 
and hence $T \equiv N$ is nilpotent. 
We drop the subscript $t=0$ in what follows.

The `only if' part is obvious because $(0,\Ker Q \cap \Ker N)$ and 
$(V,\range P + \range N)$ are both subrepresentations. 
We show the `if' part. 
Suppose that 
an $N$-invariant subspace $Y \subset W$ is contained in $\Ker Q$ 
(i.e.\ $(0,Y)$ is a subrepresentation). 
If $Y$ is nonzero, then we can take a nonzero vector $w \in Y$ 
satisfying $Nw=0$ because $N$ is nilpotent. 
Thus $w \in \Ker N \cap Y \subset \Ker N \cap \Ker Q$, 
which is a contradiction. Hence $Y=0$.
Next suppose that 
an $N$-invariant subspace $Y \subset W$ contains $\range P$ 
(i.e.\ $(V,Y)$ is a subrepresentation).
If $Y \neq W$, then consider 
the endomorphism $N_Y$ on $W/Y$ induced from $N$.
Since $N_Y$ is nilpotent and $W/Y$ is nonzero, 
$\Coker N_Y$ is nonzero. 
Thus we can take a vector $w \in W$ such that 
$w \notin Y + \range N \supset \range P + \range N$, 
which is a contradiction. Hence $Y=W$.
\end{proof}

\section{Properties of the map $\Phi_T$}\label{sec:property}

\begin{proposition}\label{prop:unique}
Suppose that two data $(V,W,T,Q,P)$ and $(V,W',T',Q',P')$ 
with the same $V \neq 0$ are both stable and 
$\Phi_T(Q,P)=\Phi_{T'}(Q',P')$. 
Then there exists an isomorphism $f \colon W \to W'$ such that
\[
Q' = Q f^{-1}, \quad P' = f P, \quad T' = f T f^{-1}.
\]
\end{proposition}

\begin{proof}
For each $t \in D$, 
set $\widehat{W}_t := W_t \oplus W'_t$ and define 
\begin{align*}
(\widehat{Q}_t,\widehat{P}_t) &:= (Q_t,P_t) \oplus (0,0) 
\in \bM(V,\widehat{W}_t)=\bM(V,W_t) \oplus \bM(V,W'_t), \\
(\widehat{Q}'_t,\widehat{P}'_t) &:= (0,0) \oplus (Q'_t,P'_t) 
\in \bM(V,\widehat{W}_t), 
\end{align*}
and
\[
\widehat{N}_t := 
\left(\, \begin{matrix} N_t & 0 \\ 0 & 0 \end{matrix} \,\right), 
\quad
\widehat{N}'_t:= 
\left(\, \begin{matrix} 0 & 0 \\ 0 & N'_t \end{matrix} \,\right)
\; \in \End(\widehat{W}_t).
\]
Then $(\widehat{Q}_t,\widehat{P}_t,\widehat{N}_t)_{t \in D}$ 
and $(\widehat{Q}'_t,\widehat{P}'_t,\widehat{N}'_t)_{t \in D}$
give points $x, x'$ in $\Rep_\QQ(V,\widehat{\bW})$
respectively, 
where $\widehat{\bW}=(\widehat{W}_t)_{t \in D}$.

It is known~\cite[Theorem 1.3]{Lus-qvar} that the ring 
$\C[\Rep_{\QQ}(V,\widehat{\bW})]^{\GL(\widehat{\bW})}$ 
of $\GL(\widehat{\bW})$-invari-ant regular functions on 
$\Rep_{\QQ}(V,\widehat{\bW})$ 
is generated by the functions
\begin{align*}
x=(\widehat{Q}_t,\widehat{P}_t,\widehat{N}_t)_{t \in D} 
&\longmapsto \chi (\widehat{Q}_s \widehat{N}_s^k \widehat{P}_s), 
\quad \chi \in \End(V)^*, k \geq 0, s \in D, \\
x=(\widehat{Q}_t,\widehat{P}_t,\widehat{N}_t)_{t \in D} 
&\longmapsto \tr \widehat{N}_s^k, \qquad k \geq 1, s \in D.
\end{align*}
Since $\widehat{N}_t,\, \widehat{N}'_t$ are nilpotent, 
the hypothesis $\Phi_T(Q,P)=\Phi_{T'}(Q',P')$ implies that 
the two points $x$ and $x'$ can not be 
distinguished by $\GL(\widehat{\bW})$-invariant functions,
in other words,
\[
\ov{\GL(\widehat{\bW}) \cdot x}\, \cap \, 
\ov{\GL(\widehat{\bW}) \cdot x'} \neq \emptyset.
\]
On the other hand, by the construction and Lemma~\ref{lem:polystable},
the above two orbits are closed. Thus $x$ and $x'$ must be isomorphic
as representations of $\QQ_{\dim V}$.
The uniqueness of the decomposition in Lemma~\ref{lem:polystable} 
implies the result.
\end{proof}

\begin{lemma}\label{lem:existence}
For any system $A(z) \in \EE_\vek (V)$, 
there exists a datum $(V,W,T,Q,P)$ satisfying $\Phi_T(Q,P) =A(z)$.
\end{lemma}

\begin{proof}
Set $\widehat{W}_t := V^{\oplus k_t}$ and 
\begin{gather*}
\widehat{Q}_t := \begin{pmatrix} A_{t,k_t} & A_{t,k_t-1} & \cdots & A_{t,1}
        \end{pmatrix} \in \Hom(\widehat{W}_t,V), \\ 
\widehat{P}_t := \left(\, \begin{matrix} 0 \\ \vdots \\ 0 \\ \unit_V
       \end{matrix} \,\right) \in \Hom(V,\widehat{W}_t), \quad
\widehat{N}_t := \left(\,\begin{matrix}
                        0     & \ \unit_V \  &        & \ 0 \\
                              & 0            & \ddots &   \\
                              &              & \ddots & \ \unit_V   \\
                        0     &              &        & \ 0 
\end{matrix} \,\right)
\in \End(\widehat{W}_t).
\end{gather*}
Then one can easily check
\[
\widehat{Q}_t \widehat{N}_t^k \widehat{P}_t = A_{t,k+1},\; k=0,1,\dots ,k_t-1,
\quad \widehat{N}_t^{k_t} =0.
\]
Thus setting 
\[
\widehat{W}:=\bigoplus \widehat{W}_t, \quad 
\widehat{T} := \bigoplus \bigl( t\,\unit_{W_t} + \widehat{N}_t\bigr) 
\in \End(W),
\]
and $(\widehat{Q},\widehat{P})
:=\bigoplus (\widehat{Q}_t,\widehat{P}_t) \in \bM(V,\widehat{W})$, 
we obtain the result.
\end{proof}

\begin{remark}\label{rem:jet}
The datum defined above has the following meaning.
The identification $\C^{k_t} \simeq S_{k_t}$ given by the basis 
$\{\, z^{k_t-1}, z^{k_t-2}, \dots ,1 \,\}$ induces isomorphisms
\begin{gather*}
\widehat{W}_t \simeq \C^{k_t} \otimes_\C V 
\simeq S_{k_t} \otimes_\C V, \\
\Hom_\C (V,\widehat{W}_t) 
\simeq V^* \otimes_\C V \otimes_\C S_{k_t}
\simeq \End_\C(V) \otimes_\C S_{k_t},\\
\Hom_\C (\widehat{W}_t,V) 
\simeq \End_\C(V) \otimes_\C S_{k_t}^* 
\simeq \End_\C(V) \otimes_\C S^{k_t},\\
\End_\C (\widehat{W}_t) \simeq \End_\C(V) 
\otimes_\C \End_\C (S_{k_t}).
\end{gather*}
Under these identifications, we can write 
\begin{align*}
\widehat{Q}_t &= \sum_{k=1}^{k_t} A_{t,k} \otimes_\C z^{-k} 
\in \Hom_\C (\widehat{W}_t,V), \\
\widehat{P}_t &= \unit_V \otimes_\C 1 
\in \Hom_\C (V,\widehat{W}_t), \\
\widehat{N}_t &= \unit_V \otimes_\C z\,\unit_{S_{k_t}} 
\in \End_\C(\widehat{W}_t).
\end{align*}
The above description was also used by Woodhouse~\cite{Wood}.
Note that looking at $\widehat{N}_t$ in particular,
we have
\[
\{\, L \in \End_\C(\widehat{W}_t) \mid 
L\widehat{N}_t = \widehat{N}_t L \,\}
= \End_{S_{k_t}}(\widehat{W}_t) = \g_{k_t}(V).
\]
\end{remark}

\begin{remark}\label{rem:oshima-normal}
For any nilpotent endomorphism $N \in \End(W)$, 
there exist a decomposition $W=\bigoplus_k W_k$ and 
injections $\iota \colon W_k \hookrightarrow W_{k-1}$ such that 
\begin{equation}\label{eq:oshima-normal}
N = \left(\, \begin{matrix}
                      0     & \quad \iota\quad &        & \quad 0      \\
                            & 0     & \ddots &        \\
                            &       & \ddots & \quad \iota   \\
                      0     &       &        & \quad 0 \end{matrix} \,\right).
\end{equation}
Such a normal form was effectively used by Oshima~\cite{Oshima-conj,Oshima-conn}.
\end{remark}

\begin{proposition}\label{prop:existence}
For any system $A(z) \in \EE_\vek (V)$ with $V \neq 0$, 
there exists a stable datum $(V,W,T,Q,P)$ 
satisfying $\Phi_T(Q,P) =A(z)$.
\end{proposition}

\begin{proof}
Take a datum $(V,W,T,Q,P)$ satisfying $\Phi_T(Q,P) =A(z)$.
We show it is stable if $\dim W$ is minimal among all such data. 
Assume that it is not stable, so we have 
a subrepresentation $(X,Y)$ such that:
\begin{enm}
\item[(a)] $X=0$ and $Y \neq 0$, or 
\item[(b)] $X=V$ and $Y \neq W$.
\end{enm}
First assume $X=0$ and $Y \neq 0$. 
Then $Y$ satisfies $Q(Y)=0$ and $T(Y) \subset Y$.
Thus the datum $(V,W,T,Q,P)$ induces a datum $(V,W/Y,T',Q',P')$ 
in the obvious way. 
Clearly we have 
\[
Q ( z\,\unit_W - T)^{-1} P = Q' ( z\,\unit_{W/Y} - T')^{-1} P',
\]
which contradicts the assumption that $\dim W$ is minimal.
Next assume that $X=V$ and $Y \neq W$.
Then $Y$ satisfies $\range P \subset Y$ and $T(Y) \subset Y$.
Thus the datum $(V,W,T,Q,P)$ induces a datum $(V,Y,T|_Y,Q|_Y,P)$ 
in the obvious way, and clearly 
\[
Q ( z\,\unit_W - T)^{-1} P = Q|_Y ( z\,\unit_Y - T|_Y)^{-1} P,
\]
which contradicts the assumption again. 
Hence $(V,W,T,Q,P)$ is stable.
\end{proof}

In fact, for given $A(z) \in \EE_\vek (V)$, 
we can also construct explicitly a stable datum 
satisfying $\Phi_T(Q,P)=A(z)$ as follows.
Let $(V,\widehat{W},\widehat{T},\widehat{Q},\widehat{P})$ be 
the datum defined in the proof of Lemma~\ref{lem:existence}, 
and set
\footnote{The definition of the matrix $\widehat{A}_t$ 
was suggested by Takemura.}
\begin{equation}\label{eq:matrix}
\widehat{A}_t := 
\left(\, \begin{matrix}
                      A_{t,k_t}     & A_{t,k_t-1} & \cdots & A_{t,1} \\
                            & A_{t,k_t}     & \ddots & \vdots \\
                            &       & \ddots & A_{t,k_t-1}  \\
                      0     &       &        & A_{t,k_t} \end{matrix} \,\right) 
\in \End(\widehat{W}_t).
\end{equation}
Then one can easily see
\[
\widehat{A}_t \widehat{N}_t = \widehat{N}_t \widehat{A}_t, \quad
\widehat{Q}_t = \begin{pmatrix} \unit_V & 0 & \cdots & 0 \end{pmatrix} \widehat{A}_t.
\]
The second relation implies $\Ker \widehat{A}_t \subset \Ker \widehat{Q}_t$. 
Thus setting $W':=\bigoplus_t \Ker \widehat{A}_t$, 
we see that the pair $(0,W')$ is a subrepresentation of the datum
$(V,\widehat{W},\widehat{T},\widehat{Q},\widehat{P})$.
Now let $(V,W,T,Q,P)$ be the quotient datum of 
$(V,\widehat{W},\widehat{T},\widehat{Q},\widehat{P})$ 
by $(0,W')$, namely, 
$W:=\widehat{W}/W'=\bigoplus \widehat{W}_t/\Ker \widehat{A}_t$
and $T,Q,P$ are the maps induced from 
$\widehat{T},\widehat{Q},\widehat{P}$ respectively.

\begin{definition}\label{dfn:canonical}
The datum $(V,W,T,Q,P)$ given above is called  
the {\em canonical datum} for the system $A(z) \in \EE_\vek(V)$.
\end{definition}

\begin{proposition}\label{prop:explicit}
The canonical datum for any $A(z) \in \EE_\vek(V)$ with $V \neq 0$ is stable.
\end{proposition}

\begin{proof}
We use Lemma~\ref{lem:stability}.
The condition $\range P_t + \range N_t = W_t$ is equivalent to 
$\range \widehat{P}_t + \range \widehat{N}_t + \Ker \widehat{A}_t = \widehat{W}_t$, 
which immediately follows from 
the definitions of $\widehat{P}_t$ and $\widehat{N}_t$.
To prove $\Ker Q_t \cap \Ker N_t =0$, 
suppose that $\hat{w} \in \widehat{W}_t$ 
represents some element in $\Ker Q_t \cap \Ker N_t$.
Then we have
\[
\widehat{Q}_t \hat{w}=0, \quad \widehat{A}_t \widehat{N}_t \hat{w} =0.
\]
The first equation implies 
\[
\widehat{P}_t \begin{pmatrix} \unit_V & 0 & \cdots & 0 \end{pmatrix} \widehat{A}_t \hat{w}
= \widehat{P}_t \widehat{Q}_t \hat{w}=0,
\]
and on the other hand, it is easy to see that
\[
\det \left[ \widehat{P}_t \begin{pmatrix} \unit_V & 0 & \cdots & 0 \end{pmatrix} + \widehat{N}_t \right] \neq 0.
\]
Since $\widehat{N}_t \widehat{A}_t \hat{w} = \widehat{A}_t \widehat{N}_t \hat{w} =0$, 
we obtain $\widehat{A}_t \hat{w} =0$. Hence $\Ker Q_t \cap \Ker N_t =0$.

By Lemma~\ref{lem:stability}, the datum is stable.
\end{proof}

\begin{remark}
On the viewpoint mentioned in Remark~\ref{rem:jet}, 
the matrix $\widehat{A}_t$ is written as
\[
\widehat{A}_t = \sum_{k=1}^{k_t} 
A_{t,k} \otimes_\C z^{k_t-k}\,\unit_{S_{k_t}} 
\in \End_\C(V) \otimes_\C \End_\C(S_{k_t}) = \End_\C (\widehat{W}_t).
\]
\end{remark}

\begin{remark}\label{rem:canonical}
In what follows we assume $D=\{ 0 \}$ and omit the subscript $t=0$.

(a) The matrix $\widehat{A} \in \End(\widehat{W})$ is invertible 
if and only if the top coefficient $A_k$ of $A(z)$ is invertible.

(b) Set $d:=\ord(A)$ (see \eqref{eq:order} for the definition) 
and identify $\C^k$ with $S_k$ as in Remark~\ref{rem:jet}.
Then it is easy to see that $\widehat{A}$ vanishes on 
the subspace $V \otimes z^d S_k$ 
and so induces a homomorphism
\[
V \otimes S_d \simeq V \otimes (S_k/z^d S_k) 
\xrightarrow{\widehat{A}} V \otimes (S_k/z^d S_k) \simeq V \otimes S_d.
\]
Clearly it coincides with 
the matrix $\widehat{A} \in \End(V \otimes \C^d)$ 
constructed from $A(z)$ 
regarded as an element of $\EE_d(V)$.
Hence the construction of the canonical datum does not depend on the choice 
of $k \geq \ord(A)$.

(c) For $i=1,2$, let $V^i$ be a nonzero finite-dimensional $\C$-vector space, 
$A^i(z) \in \EE_k(V^i)$, and
$(V^i,W^i,N^i,Q^i,P^i)$ be the canonical datum for $A^i(z)$.
Then the canonical datum for the direct sum 
$A(z):=A^1(z) \oplus A^2(z) \in \EE_k(V^1 \oplus V^2)$ is naturally 
identified with
\[
\bigoplus_{i=1,2} (V^i,W^i,N^i,Q^i,P^i) 
=(V^1 \oplus V^2,W^1 \oplus W^2,N^1 \oplus N^2,Q^1 \oplus Q^2,P^1 \oplus P^2).
\]
More generally, if $(V^i,W^i,N^i,Q^i,P^i),\, i=1,2$ 
are stable data with $V^i \neq 0$, then the above direct sum is also stable 
and 
\[
\Phi_{N^1 \oplus N^2} (Q^1 \oplus Q^2,P^1 \oplus P^2) = 
\Phi_{N^1}(Q^1,P^1) \oplus \Phi_{N^2}(Q^2,P^2).
\]
\end{remark}

For a subset $X$ of $\bM(V,W)$ with $V \neq 0$ and 
an endomorphism $T \in \End(W)$ satisfying \eqref{eq:T}, 
we set
\[
X^{T\st} := \{\,(Q,P) \in X \mid
\text{$(V,W,T,Q,P)$ is stable}\,\}.
\]
One can easily see that 
$\bM(V,W)^{T\st}$ is invariant 
under the action of the centralizer $G_T$ of $T$ 
(Note that $G_T \subset \GL(\bW)$).

\begin{proposition}\label{prop:geom}
Suppose that $V \neq 0$ and $T \in \End(W)$ satisfies \eqref{eq:T}.

{\rm (a)} The $G_T$-action on $\bM(V,W)^{T\st}$ is free and proper.

{\rm (b)} The map $\Phi_T$ induces a Poisson embedding
\[
\bM(V,W)^{T\st} /G_T \hookrightarrow \EE_\vek (V).
\]

{\rm (c)} If 
$\Phi_T(\bM(V,W)^{T\st}) \cap \Phi_{T'}(\bM(V,W')^{T'\st}) \neq \emptyset$, 
then there exists an isomorphism $f \colon W \to W'$ such that 
$T'=fTf^{-1}$.

{\rm (d)} The subset $\bM(V,W)^{T\st}$ is invariant 
under the action of $G_\vek (V)$.
\end{proposition}

\begin{proof}
(a) Considering the $G_T$-equivariant closed embedding
\[
\varphi \colon \bM(V,W) \hookrightarrow \Rep_\QQ(V,\bW),
\quad (Q,P) \mapsto (Q_t,P_t,N_t)_{t \in D},
\]
we see that the action is free by Lemma~\ref{lem:stable-free}.
To see properness, 
we use an identification 
$\Rep_\QQ(V,\bW) \simeq \Rep_{\QQ_{\dim V}}(\C,\bW)$ 
and consider the subset of irreducible representations 
$\Rep_{\QQ_{\dim V}}^\text{irr}(\C,\bW)$.
We have
\[
\varphi\bigl(\bM(V,W)^{T\st}\bigr) = 
\varphi\bigl(\bM(V,W)\bigr) \cap 
\Rep_{\QQ_{\dim V}}^\text{irr}(\C,\bW).
\]
By a standard fact in the geometric invariant theory
~\cite[Corollary 2.5]{GIT} 
together with King's work~\cite{King}, 
the $\GL(\bW)$-action on 
$\Rep_{\QQ_{\dim V}}^\text{irr}(\C,\bW)$ is proper. 
Therefore the above implies the properness of the $G_T$-action 
on $\bM(V,W)^{T\st}$.

(b) and (c) follow from Proposition~\ref{prop:unique}. 

(d) The explicit description of the action \eqref{eq:action} 
shows that 
\begin{align*}
\Ker (g \cdot Q)_t \cap \Ker N_t &= 
\Ker Q_t \cap \Ker N_t, \\
\range (g \cdot P)_t + \range N_t &=
\range P_t + \range N_t
\end{align*}
for any $(Q,P) \in \bM(V,W)$, $g \in G_\vek(V)$ and $t \in D$. 
Therefore Lemma~\ref{lem:stability} implies the result.
\end{proof}

Let $\g_T$ be the Lie algebra of $G_T$ and 
$p_T \colon \gl(W) \to \g_T^*$ be 
the transpose of the inclusion $\g_T \hookrightarrow \gl(W)$.
Recall that the map 
\[
\mu_W \colon \bM(V,W) \to \gl(W); 
\quad (Q,P) \mapsto -PQ
\]
is a moment map generating the $\GL(W)$-action. 
We set
\[
\mu_T := p_T \circ \mu_W \colon \bM(V,W) \to \g_T^*, 
\]
which is a $G_T$-moment map. 

\begin{lemma}\label{lem:invariant}
The map $\mu_T$ is $G_\vek (V)$-invariant.
\end{lemma}

\begin{proof}
For each $t \in D$, let $G_{N_t} \subset \GL(W_t)$ 
be the centralizer of $N_t$, 
$\g_{N_t}$ be its Lie algebra, 
and $p_t \colon \gl(W_t) \to \g_{N_t}^*$ 
be the transpose of the inclusion 
$g_{N_t} \hookrightarrow \gl(W_t)$. 
Then obviously we have $p_t(A_t N_t)=p_t(N_t A_t)$ for any $A_t \in \gl(W_t)$, 
which implies  
\begin{align*}
p_t \bigl( (g_t \cdot P_t)(g_t \cdot Q_t) \bigr) &= 
p_t \left( \sum_{k,l \geq 0} N_t^k P_t (g_t^{-1})_k g_{t,l} Q_t N_t^l \right)
\\ &= 
p_t \left( \sum_{k,l \geq 0} N_t^{k+l} P_t (g_t^{-1})_k g_{t,l} Q_t \right).
\end{align*}
Substituting the equality 
$g_t^{-1}(z) g_t (z) =1$ (mod $z^{k_t}$) into the above, we have
\[
p_t \bigl( (g_t \cdot P_t)(g_t \cdot Q_t) \bigr) =
p_t \left( P_t Q_t \right).
\]
Since 
\[
\mu_T(Q,P) = -\bigl[ p_t (P_t Q_t) \bigr]_{t \in D} 
\in \bigoplus_t \g_{N_t}^* = \g_T,
\]
we obtain the result.
\end{proof}

\begin{lemma}\label{lem:duality}
Suppose $V \neq 0$.
Then the tangent space of the $G_\vek(V)$-orbit 
through any $x \in \bM(V,W)^{T\st}$ coincides with
$\Ker d_x \mu_T$, and its dimension is constant on $\bM(V,W)^{T\st}$.
\end{lemma}

\begin{proof}
The moment map equations for $\Phi_T$ and $\mu_T$ imply
\[
\Ker d_x \mu_T = T_x (G_T \cdot x)^\omega, \quad
\Ker d_x \Phi_T = T_x \bigl( G_\vek(V) \cdot x \bigr)^\omega,
\]
where the superscript $\omega$ means the symplectic orthogonal complement subspace. 
On the other hand, Proposition~\ref{prop:geom}, (b) implies 
\[
\Ker d_x \Phi_T = T_x (G_T \cdot x).
\]
Hence 
\[
\Ker d_x \mu_T = T_x (G_T \cdot x )^\omega 
= (\Ker d_x \Phi_T)^\omega = T_x \bigl( G_\vek(V) \cdot x \bigr),
\]
whose dimension is constant on $\bM(V,W)^{T\st}$ by Proposition~\ref{prop:geom}, (a).
\end{proof}

\begin{remark}\label{rem:duality}
The above lemma implies that for any $x \in \bM(V,W)^{T\st}$, 
the subspaces $\Ker d_x \mu_T$ and $\Ker d_x \Phi_T$ 
are symplectic orthogonal complement to each other, 
namely the pair $(\mu_T, \Phi_T)$ is a {\em dual pair} of moment maps 
in the sense of Weinstein~\cite{Wein}, 
as mentioned by Harnad~\cite{Harnad}.
\end{remark}

\begin{theorem}\label{thm:geom}
Suppose $V \neq 0$.
Then for any $G_\vek (V)$-coadjoint orbit $\bO$, 
there exist a finite-dimensional $\C$-vector space $W$, 
an endomorphism $T$ of $W$, and 
a $G_T$-coadjoint orbit $\OO \subset \g_T^*$ such that 
the map $\Phi_T$ induces a symplectomorphism 
\[
\mu_T^{-1}(\OO)^{T\st} /G_T \xrightarrow{\simeq} \bO.
\]
Moreover the choice of $(W, T, \OO)$ 
is unique in the following sense: if another triple 
$(W',T',\OO')$ has the same properties, then 
there exists an isomorphism $f \colon W \to W'$ such that
$fTf^{-1}=T'$ and $f \OO f^{-1}=\OO'$.
\end{theorem}

\begin{proof}
Take any $A(z) \in \bO$ and then 
a stable datum $(V,W,T,Q,P)$ satisfying $A(z)=\Phi_T(Q,P)$. 
Let $\OO \subset \g_T^*$ be 
the coadjoint orbit through $\mu_T(Q,P)$. 
Then by Proposition~\ref{prop:geom},~(d) and Lemma~\ref{lem:invariant}, 
the subset $\mu_T^{-1}(\OO)^{T\st}$ is $G_\vek(V)$-invariant 
and hence by equivariance, 
its image under $\Phi_T$ is a $G_\vek(V)$-invariant subset 
of $\g_\vek^*(V)$ containing $\bO$.
To see that the image actually coincides with $\bO$, 
we have to show that the induced $G_\vek(V)$-action on 
$\mu_T^{-1}(\OO)^{T\st}/G_T$ is transitive, or equivalently, 
the restricted $G_\vek(V)$-action on $\mu_T^{-1}(\alpha)^{T\st}$, 
where $\alpha \in \OO$, is transitive. 

Proposition~\ref{prop:geom}, (a) implies that 
$\mu_T^{-1}(\alpha)^{T\st}$ is a pure-dimensional smooth subvariety, 
and in which, by Lemma~\ref{lem:duality}, 
any $G_\vek(V)$-orbit is a Zariski open subset. 
Therefore in order to show the transitivity of the action, it is sufficient to show that 
$\mu_T^{-1}(\alpha)^{T\st}$ is an irreducible variety. 

To see that $\mu_T^{-1}(\alpha)^{T\st}$ is irreducible,
consider the first projection 
$\varphi \colon \mu_T^{-1}(\alpha)^{T\st} \to \Hom(W,V)$.
By the definition of $\mu_T$ and Lemma~\ref{lem:stability}, 
we see that any nonempty fiber of $\varphi$ is 
a Zariski open subset of an affine space. 
Moreover Lemma~\ref{lem:transitive} below shows that 
$\varphi$ is a dominant morphism and 
every nonempty fibers have the same dimension, 
because any $g \in G_\vek(V)$ induces an isomorphism between 
the fibers $\varphi^{-1}(Q)$ and $\varphi^{-1}(g \cdot Q)$. 
Therefore we may apply the following fact (c.f.~\cite[Lemma 6.1]{Cra-geom}) 
to deduce that $\mu_T^{-1}(\alpha)^{T\st}$ is irreducible:
\begin{quotation}\em
If $X$ is a pure-dimensional scheme, 
$Y$ is an irreducible scheme and
$f \colon X \to Y$ is a dominant morphism with all fibers 
irreducible of constant dimension, then $X$ is irreducible.
\end{quotation} 

The uniqueness assertion immediately 
follows from Proposition~\ref{prop:unique}.
\end{proof}

\begin{lemma}\label{lem:transitive}
Let $V,W$ be two $\C$-vector spaces and let $N \in \End(W)$ 
be a nilpotent endomorphism of $W$ with $N^k=0$. 
Define an action of $G_k(V)$ on $\Hom(W,V)$ as in \eqref{eq:action}.
Then the restricted action on the subset 
\[
\{\, Q \in \Hom(W,V) \mid \text{$Q |_{\Ker N}$ is injective}\,\}
\]
is transitive.
\end{lemma}

\begin{proof}
Let $Q, Q' \in \Hom(W,V)$ and 
assume that both $Q |_{\Ker N}$ and $Q'|_{\Ker N}$ are injective. 
Then we solve the equation
\begin{equation}\label{eq:transitive}
\sum_{i=0}^{k-1} g_i Q N^i = Q', \quad g(z)=\sum_i g_i z^i \in G_k (V).
\end{equation}
First we restrict the both sides to $\Ker N$. 
Then we have
\[
g_0 Q |_{\Ker N} = Q' |_{\Ker N}.
\]
By the assumption we can find $g_0 \in \GL(V)$ satisfying the above.
Next, we restrict the both sides of \eqref{eq:transitive} to $\Ker N^2$. Then we have
\[
(g_0 Q + g_1 Q N) |_{\Ker N^2} = Q' |_{\Ker N^2},
\]
or equivalently,
\[
g_1 Q N |_{\Ker N^2} = (Q' -g_0 Q )|_{\Ker N^2}.
\]
Because $Q |_{\Ker N}$ is injective, the kernel of 
$Q N |_{\Ker N^2}$ is just $\Ker N$. 
Also, we have $Q' - g_0 Q |_{\Ker N} =0$. 
Hence the both sides of the above equation descend to homomorphisms 
from $\Ker N^2 /\Ker N$;
\[
g_1 Q N |_{\Ker N^2 /\Ker N} = (Q' -g_0 Q )|_{\Ker N^2 /\Ker N},
\]
and we can find $g_1 \in \End(W)$ satisfying the above 
as $Q N |_{\Ker N^2 /\Ker N}$ is injective.

Iterating this argument inductively, 
we find $g_i \in \End(W),\, i=0,1, \dots, k-1$ 
with $\det g_0 \neq 0$, satisfying the equation 
\[
g_i Q N^i |_{\Ker N^{i+1}} 
= \left. \left( Q' - \sum_{j=0}^{i-1} g_j Q N^j \right) \right|_{\Ker N^{i+1}},
\]
and finally we obtain a desired $g(z) \in G_k(V)$.
\end{proof}

\section{Harnad duality}\label{sec:duality}

In this section we formulate the `Harnad dual' in a categorical setting, 
and introduce some important properties of it.

\begin{definition}
We call a sextuple $(V,W,S,T,Q,P)$ consisting of:
\begin{itm}
\item two finite-dimensional $\C$-vector spaces $V,W$;
\item $(S,T) \in \End(V) \oplus \End(W)$; and
\item $(Q,P) \in \bM(V,W)$,
\end{itm}
as a {\em Harnad datum}.
\end{definition}

We always assume our Harnad data $(V,W,S,T,Q,P)$ 
satisfy Condition~\eqref{eq:T} and
\begin{equation}\label{eq:S}
\prod_{s \in E} (S-s\,\unit_V)^{l_s}=0,
\end{equation}
where the zero datum $(0,0,0,0,0,0)$ is understood to
satisfy these two conditions.

Harnad data $(V,W,S,T,Q,P)$ 
(satisfying \eqref{eq:T} and \eqref{eq:S})
form an abelian category, which we denote by $\lHk$.
A morphism from $(V,W,S,T,Q,P)$ to $(V',W',S',T',Q',P')$
is a pair $(f_V,f_W)$ of linear maps 
$f_V \colon V \to V'$ and $f_W \colon W \to W'$
satisfying
\begin{gather*}
\begin{aligned}
f_V S &= S' f_V, \\
f_V Q &= Q' f_W, 
\end{aligned}\quad
\begin{aligned}
f_W T &= T' f_W, \\
f_W P &= P' f_V.
\end{aligned}
\end{gather*}

Now consider systems of linear ordinary differential equations of the form
\[
\frac{du}{dz}=A(z)u, \quad
A(z) = S + \sum_{t \in D} \sum_{k=1}^{k_t} \frac{A_{t,k}}{(z-t)^k} 
\in \gl(V) \oplus \EE_\vek(V),
\]
such that $S=\lim_{z \to \infty} A(z)$ satisfies \eqref{eq:S}.
We define an abelian category $\lDk$ by 
\begin{itm}
\item an object of $\lDk$ is a pair $(V,A)$ consisting of
a finite-dimensional $\C$-vector space $V$ and  
a system $A(z) \in \gl(V) \oplus \EE_\vek (V)$ 
such that $\lim_{z \to \infty} A(z)$ satisfies \eqref{eq:S};
\item a morphism from $(V,A)$ to $(V',A')$, 
where $A(z)=S+\sum A_{t,k}(z-t)^{-k},\, A'(z)=S'+\sum A'_{t,k}(z-t)^{-k}$, 
is a linear map 
$f \colon V \to V'$ satisfying $f \circ S = S' \circ f$ and
$f \circ A_{t,k} = A'_{t,k} \circ f$ for all $t,k$.
\end{itm}
Then it is easy to see that
\[
\Phi \colon (V,W,S,T,Q,P) \mapsto \bigl( V,S + \Phi_T(Q,P) \bigr) 
=\bigl( V,S+Q(z\,\unit_W-T)^{-1}P \bigr)
\]
defines a functor from $\lHk$ to $\lDk$ 
(the map between the sets of morphisms 
is given by the projection $(f_V,f_W) \mapsto f_V$).
The notion of canonical datum 
(Definition~\ref{dfn:canonical}) gives a `section' of $\Phi$;
for $(V,A) \in \lDk$,
letting $S :=\lim_{z \to \infty} A(z)$ 
and $(V,W,T,Q,P)$ be the canonical datum for 
the $\EE_\vek(V)$-component of $A(z)$, set
\[
\kappa(V,A) := (V,W,S,T,Q,P).
\]
By the construction of the canonical datum, 
it then defines a functor from $\lDk$ to $\lHk$ 
and satisfies $\Phi \circ \kappa = \unit$.
Note that there is an equivalence of categories
\[
\sigma \colon \lHk \to \kHl; \quad (V,W,S,T,Q,P) \mapsto (W,V,T,S,P,Q),
\]
which together with $\Phi$ induces
\begin{align*}
\Phi \circ \sigma (V,W,S,T,Q,P) 
&= \bigl( W,T + P(\zeta\,\unit_V -S)^{-1}Q \bigr) \\ 
&= \bigl( W, T-\Psi_S(Q,P) \bigr),
\end{align*}
where $\zeta$ denotes the indeterminate for $\kDl$.

\begin{definition}\label{dfn:dual}
We call 
\[
\HD := \Phi \circ \sigma \circ \kappa \colon \kDl \to \lDk
\]
as the {\em Harnad dual} functor.
\end{definition}

Here we give two simple examples.

\begin{example}\label{ex:dual}
(a) If any $\EE_{k_t}(V)$-component of 
$A(z)$ is zero (i.e., $A(z)=S$), 
then the corresponding canonical datum is zero, whence $\HD(V,S)=(0,0)$.

(b) Let us compute the Harnad dual $\HD(V,A)$ 
in the case $(V,A)=(\C,s+\alpha)$ for some 
$s \in E,\,\alpha(z)=\sum \alpha_{t,j} (z-t)^{-j} \in \EE_\vek(\C)$.
Set $d_t:=\ord(\alpha_t)$ and 
let $(\C,W,T,Q,P)$ be the canonical datum for $\alpha(z)$,
which defines $(W_t,N_t,Q_t,P_t)$ for each $t \in D$ as usual.
Then Remark~\ref{rem:canonical} shows that 
$W_t=\C^{d_t}$ (which is understood as zero if $d_t=0$) and 
\begin{gather*}
Q_t = \begin{pmatrix} \alpha_{t,d_t} & \alpha_{t,d_t-1} & \cdots & \alpha_{t,1}
        \end{pmatrix} \in \Hom(\C^{d_t},\C), \\ 
P_t = \left(\, \begin{matrix} 0 \\ \vdots \\ 0 \\ 1
       \end{matrix} \,\right) \in \Hom(\C,\C^{d_t}), \quad
N_t = \left(\,\begin{matrix}
                        0  & \quad 1 \quad &        & \quad 0 \\
                           & 0             & \ddots &   \\
                           &               & \ddots & \quad 1   \\
                        0  &               &        & \quad 0 
\end{matrix} \,\right)
\in \End(\C^{d_t}).
\end{gather*}
For $t, t' \in D$, let
\[
R_{t',t} := P_{t'} Q_t =
\left(\,\begin{matrix}
                        0  & 0              & \cdots &  0 \\
                    \vdots & \vdots         &        &  \vdots \\
                        0  & 0              & \cdots &  0   \\
            \alpha_{t,d_t} & \alpha_{t,d_t-1} & \cdots & \alpha_{t,1} 
\end{matrix} \,\right) \in \Hom(\C^{d_t},\C^{d_{t'}}).
\]
Then the Harnad dual $(W,B)=\HD(\C,s+\alpha)$ is described as
\[
B(\zeta) = T + \frac{R}{\zeta-s}, \quad 
R=PQ=(R_{t',t})_{t',t \in D}.
\]
\end{example}

Considering the sub-objects, we define the following:
\begin{definition}
{\rm (a)} A {\em subrepresentation} of a Harnad datum $(V,W,S,T,Q,P)$ 
is a pair $(X,Y)$ of subspaces $X \subset V,\, Y \subset W$
satisfying 
\[
S(X) \subset X, \quad T(Y) \subset Y, \quad
Q(Y) \subset X, \quad P(X) \subset Y.
\]
A nonzero Harnad datum $(V,W,S,T,Q,P)$ is said to be {\em irreducible}
if it has no nonzero proper subrepresentations.

{\rm (b)} A pair $(V,A) \in \lDk$ with $V \neq 0$ 
is said to be {\em irreducible}
if $V$ has no nonzero proper subspace 
preserved by all $A_{t,k}$ and $S$. 
\end{definition}

Any subrepresentation of a Harnad datum $(V,W,S,T,Q,P)$ 
is also that of the datum $(V,W,T,Q,P)$. 

\begin{lemma}\label{lem:irr-st}
If a Harnad datum $(V,W,S,T,Q,P) \in \lHk$ is irreducible and $V \neq 0$, 
then the datum $(V,W,T,Q,P)$ is stable.
\end{lemma}

\begin{proof}
Suppose that a subrepresentation $(X,Y)$ of the datum $(V,W,T,Q,P)$ 
satisfies $X=0$ or $X=V$. Then $S(X) \subset X$, and hence 
the pair $(X,Y)$ is a subrepresentation of $(V,W,S,T,Q,P)$.
Therefore the irreducibility of $(V,W,S,T,Q,P)$ implies 
the stability of $(V,W,T,Q,P)$.
\end{proof}

\begin{lemma}\label{lem:irred}
Let $(V,W,S,T,Q,P) \in \lHk$ with $V \neq 0$ and 
suppose that $(V,W,T,Q,P)$ is stable. 
Then $(V,W,S,T,Q,P)$ is irreducible if and only if 
$\Phi(V,W,S,T,Q,P) \in \lDk$ is irreducible.
In particular, a pair $(V,A) \in \lDk$ with $V \neq 0$ is irreducible 
if and only if 
$\kappa(V,A) \in \lHk$ is irreducible.
\end{lemma}

\begin{proof}
We use the normal form mentioned in Remark~\ref{rem:oshima-normal}, namely,
take decompositions $W_t= \bigoplus_k W_{t,k}$ 
together with injections $\iota \colon W_{t,k} \to W_{t,k-1}$
such that $N_t$ is written as in \eqref{eq:oshima-normal}. 
We regard $W_{t,k} \subset W_{t,k-1}$ using $\iota$.
Note that if we denote by $Q_{t,k}$ (resp.\ $P_{t,k}$) the block components 
of $Q_t$ (resp.\ $P_t$) with respect to the decomposition 
$W_t=\bigoplus_k W_{t,k}$, 
the system $A(z):=S+\Phi_T(Q,P)$ satisfies
\[
A_{t,k} = \sum_{j \geq 1} Q_{t,j} P_{t,j+k-1}.
\]

Now we show the `only if' part.
Suppose that a subspace 
$X \subset V$ is preserved by all $A_{t,k}$ and $S$. 
Then we set
\[
Y_{t,k} := \sum_{j \geq k} P_{t,j}(X) \subset W_{t,k}, 
\quad Y_t := \bigoplus_k Y_{t,k} \subset W_t.
\]
Obviously we have 
$P_t(X) \subset Y_t$ and $N_t(Y_t) \subset Y_t$. 
Moreover, we have
\begin{align*}
Q_t(Y_t) &= \sum_{k\geq 1} Q_{t,k}(Y_{t,k}) \\
&= \sum_{k \geq 1} \sum_{j \geq k} Q_{t,k}P_{t,j}(X) \\
&= \sum_{k \geq 1} A_{t,k}(X) \subset X.
\end{align*}
Thus we see that 
the pair $(X,Y)$, where $Y=\bigoplus Y_t$, 
is a subrepresentation of the Harnad datum $(V,W,S,T,Q,P)$. 
Because $(V,W,S,T,Q,P)$ is irreducible, we have $X=0$ or $X=V$.
Hence $(V,A)$ is irreducible.

Next we show the `if' part.
Suppose that a subrepresentation $(X,Y)$ 
of the Harnad datum $(V,W,S,T,Q,P)$ is given.
Then we have $S(X) \subset X$ and 
\[
A_{t,k}(X) = Q_t N_t^{k-1} P_t (X) \subset X.
\]
Thus by the irreducibility of $(V,A)$ we get $X=0$ or $X=V$. 
Because the pair $(X,Y)$ is also a subrepresentation of 
the datum $(V,W,T,Q,P)$ which is stable, 
this implies $Y=0$ or $Y=W$ respectively.
Hence $(V,W,S,T,Q,P)$ is irreducible.
\end{proof}

\begin{lemma}\label{lem:canonical}
Let $(V,W,S,T,Q,P) \in \lHk$ with $V \neq 0$ 
and suppose that $(V,W,T,Q,P)$ is stable.
Then $(V,W,S,T,Q,P)$ and its image under $\kappa \circ \Phi$ are 
isomorphic as objects in $\lHk$;
\[
\kappa \circ \Phi (V,W,S,T,Q,P) \sim (V,W,S,T,Q,P).
\]
\end{lemma}

\begin{proof}
Note that $\kappa \circ \Phi$ effects no change in both $V$ and $S$.
Therefore Proposition~\ref{prop:unique} together with 
Proposition~\ref{prop:explicit}
gives a desired isomorphism of the form $(\unit_V,f)$.
\end{proof}

\begin{theorem}\label{thm:duality-cat}
If $(V,A) \in \lDk$ is irreducible and $(V,A) \not\sim (\C,s)$ for any $s \in E$,
then $\HD(V,A)$ is also irreducible and
\[
\HD \circ \HD (V,A) \sim (V,A).
\]
\end{theorem}

\begin{proof}
The assumption together with
Lemma~\ref{lem:irred} implies that 
$\kappa(V,A) \in \lHk$ is irreducible, 
and so is $\sigma \circ \kappa(V,A)$ 
since $\sigma$ clearly preserves the irreducibility.

Now set $(V,W,S,T,Q,P):=\kappa(V,A)$.
If $W \neq 0$, applying Lemma~\ref{lem:irr-st} to $\sigma \circ \kappa(V,A)$ 
shows that $(W,V,S,P,Q)$ is stable.
Thus we see from Lemma~\ref{lem:irred} that
$\HD(V,A)=\Phi \circ \sigma \circ \kappa(V,A)$ is irreducible,
and hence by Lemma~\ref{lem:canonical} we have
\begin{align*}
\HD \circ \HD (V,A) &= \Phi \circ \sigma \circ (\kappa \circ 
\Phi) \circ \sigma \circ \kappa (V,A) \\
&\sim \Phi \circ (\sigma \circ \sigma) \circ \kappa (V,A) \\
&= \Phi \circ \kappa (V,A) = (V,A).
\end{align*}
If $W=0$, the construction of the canonical datum shows that $A(z)=S$.
Then the irreducibility of $(V,A)=(V,S)$ implies that $V=\C$
and $S$ is a scalar satisfying \eqref{eq:S}.
\end{proof}

The following is an immediate consequence of the above theorem:
\begin{corollary}\label{cor:exceptional}
An irreducible pair $(V,A) \in \lDk$ satisfies $\HD(V,A)=(0,0)$ 
if and only if $(V,A) \sim (\C,s)$ for some $s \in E$.
\end{corollary}

For a subset $X$ of $\bM(V,W)$ with $W \neq 0$ and 
an endomorphism $S \in \End(V)$ satisfying \eqref{eq:S}, we set
\[
X_{S\st} := \{\, (Q,P) \in \bM(V,W) \mid 
\text{$(W,V,S,P,Q)$ is stable}\,\}.
\] 
Let $G_S \subset \GL(V)$ be the centralizer of $S$, 
$\g_S$ be its Lie algebra, 
and $p_S \colon \gl(V) \to \g^*_S$ be 
the transpose of the inclusion $\g_S \hookrightarrow \gl(V)$.
Then the composite $\mu_S := p_S \circ \mu_V \colon \bM(V,W) \to \g^*_S$ 
is a moment map generating the $G_S$-action.

\begin{theorem}\label{thm:duality-geom}
Let $(V,W,S,T,Q,P)$ be an irreducible Harnad datum with $V, W \neq 0$. 
Let 
\begin{itm}
\item $\bO_V$ be the $G_\vek(V)$-coadjoint orbit 
through $\Phi_T(Q,P)$;
\item $\bO_W$ be the $G_\vel(W)$-coadjoint orbit 
through $\Psi_S(Q,P)$;
\item $\OO_T$ be the $G_T$-coadjoint orbit 
through $\mu_T(Q,P)$; and
\item $\OO_S$ be the $G_S$-coadjoint orbit through $\mu_S(Q,P)$.
\end{itm}
Then the two spaces 
\begin{gather*}
\Mreg_S(\bO_V,\OO_S) := 
\rset{A(z)\in S+\bO_V}{
\begin{aligned}
&\text{\rm $(V,A)$ is irreducible}, \\
&p_S \Bigl( \Res_{z=\infty} A(z) \Bigr) \in -\OO_S
\end{aligned}}
/G_S,
\end{gather*}
and
\begin{gather*}
\Mreg_{-T}(\bO_W,\OO_T) := 
\rset{B(\zeta)\in -T+\bO_W}{
\begin{aligned}
&\text{\rm $(W,B)$ is irreducible}, \\
&p_T \Bigl( \Res_{\zeta=\infty} B(\zeta) \Bigr) \in -\OO_T
\end{aligned}}/G_T,
\end{gather*}
are both holomorphic symplectic manifolds 
and symplectomorphic to each other.
The symplectomorphism is given by $(W,-B(\zeta)) \sim \HD(V,A(z))$.
\end{theorem}

\begin{proof}
By Theorem~\ref{thm:geom}, we have a 
$G_S$-equivariant symplectomorphism
\[
S+\Phi_T \colon \mu_T^{-1}(\OO_T)^{T\st}/G_T 
\xrightarrow{\simeq} S + \bO_V.
\]
Under this isomorphism,
the $G_S$-moment map $p_S \circ \mu_V$ on the left hand side 
corresponds to the map $-p_S \circ \Res_{z=\infty}$.
Thus we have a bijection
\begin{multline*}
S+\Phi_T \colon
\mu_S^{-1}(\OO_S) \cap \mu_T^{-1}(\OO_T)^{T\st}/G_S \times G_T \\
\to \rset{A(z) \in S + \bO_V}{
p_S \Bigl( \Res_{z=\infty} A(z) \Bigr) \in -\OO_S} /G_S.
\end{multline*}
Similarly, we have a bijection
\begin{multline*}
-T+\Psi_S \colon
\mu_S^{-1}(\OO_S)_{S\st} \cap \mu_T^{-1}(\OO_T)/G_S \times G_T \\
\to \rset{B(\zeta) \in -T + \bO_W}{
p_T \Bigl( \Res_{\zeta=\infty} B(\zeta) \Bigr) \in -\OO_T} /G_T.
\end{multline*}
If a system $A(z) \in S + \bO_V$ is irreducible 
and $(V,W,T,\tilde{Q},\tilde{P})$ is a stable datum satisfying 
$\Phi_T(\tilde{Q},\tilde{P})=A(z)-S$, 
then the Harnad datum $(V,W,S,T,\tilde{Q},\tilde{P})$ 
is also irreducible by Lemma~\ref{lem:irred}, 
and hence so is 
$B(\zeta)=-T + \Psi_S(\tilde{Q},\tilde{P})$ 
by Lemma~\ref{lem:irred} again. 
The result follows.
\end{proof}

\begin{remark}\label{rem:moduli}
In the situation of Theorem~\ref{thm:duality-geom},
take an arbitrary $R \in p_S^{-1}(\OO_S)$ and 
let $\bO_{V,\infty}$ be the $G_2(V)$-coadjoint orbit through 
$-Sw^{-2} -Rw^{-1}$, where we use $w$ as the indeterminate instead of $z$. 
Note that $\bO_{V,\infty}$ does not depend on the choice of $R$. 
Indeed, for any other $R' \in p_S^{-1}(\OO_S)$, 
we can find $a \in G_S$ satisfying 
\[
aRa^{-1}-R' \in \Ker p_S = \range \ad_S,
\]
namely, there exists some $b \in \End(V)$ such that 
$aRa^{-1} = R' + [S,b]$.
Then setting $g(w):= a +baw \in G_2(V)$, 
one can easily check that 
$g \cdot (-Sw^{-2}-Rw^{-1}) =-Sw^{-2}-R'w^{-1}$.

Using $\bO_{V,\infty}$, we now have the following description of 
$\Mreg_S(\bO_V,\OO_S)$: 
\begin{multline*}
\Mreg_S(\bO_V,\OO_S) \simeq \Mreg(\bO_V,\bO_{V,\infty}) := \\
\rset{(A^0(z),A_\infty(w)) \in \bO_V \times \bO_{V,\infty}}{
\begin{aligned}
&\text{$-A_{\infty,2} + A^0(z)$ is irreducible},\\
&A_{\infty,1}=\Res_{z=\infty} A^0(z)
\end{aligned}}/\GL(V),
\end{multline*}
where $A_\infty(w)=A_{\infty,1}w^{-1}+A_{\infty,2}w^{-2}$.
The symplectomorphism is given by 
\[
S+\bO_V \ni A(z) \mapsto 
\biggl( A(z)-S, -Sw^{-2}+\Bigl( \Res_{z=\infty}A(z) \Bigr) w^{-1} \biggr).
\]
The second component $-Sw^{-2}+w^{-1}\Res_{z=\infty} A(z)$ 
on the right hand side 
is equal to the principal part of the Laurent expansion of 
$A(z)dz$ at $z \equiv w^{-1} =\infty$. 
Therefore we refer to $\Mreg(\bO_V,\bO_{V,\infty})$ 
as the {\em naive moduli space of irreducible systems 
having singularities on $D \cup \{ \infty \}$ 
with truncated formal type 
$(\bO_V, \bO_{V,\infty})$}. 
Since the map 
\[
\bO_V \times \bO_{V,\infty} \ni (A^0(z),A_\infty(w)) \mapsto 
A_{\infty,1} - \Res_{z=\infty}A^0(z)
\]
is a $\GL(V)$-moment map, 
$\Mreg(\bO_V,\bO_{V,\infty})$ is an open subset 
of the symplectic quotient of $\bO_V \times \bO_{V,\infty}$ 
by the $\GL(V)$-action, 
which was studied by Boalch~\cite{Boa-thesis}.
\end{remark}

\begin{remark}\label{rem:bipart}
If both $S$ and $T$ are semisimple, 
then we have $G_S=\prod_s \GL(V_s)$ and $G_T=\prod_t \GL(W_t)$. 
In such cases the open subset of
\[
\mu_S^{-1}(\OO_S) \cap \mu_T^{-1}(\OO_T) /G_S \times G_T,
\]
given by all irreducible Harnad data with fixed $(V,W,S,T)$,
is the Nakajima quiver variety~\cite{Nak-duke1} 
associated to some graph. 
Therefore Theorem~\ref{thm:duality-geom} tells us that 
the two naive moduli spaces 
$\Mreg_S(\bO_V,\OO_S)$ and 
$\Mreg_T(\bO_W,\OO_T)$  
are symplectomorphic to the same Nakajima quiver variety.
It is a typical example of Boalch's 
`different realizations' 
(see \cite[\S 4.7]{Boa-quiver} and also \cite[Appendix A]{Boa-diff}).
\end{remark}

\begin{remark}\label{rem:laplace}
Assume $E=\{ 0 \}$ and $\vel=(1)$.
Let $(V,W,0,T,Q,P)$ be an irreducible Harnad datum with $S=0$, 
$V, W \neq 0$.
Then $(V,A)=\Phi(V,W,0,T,Q,P)$ and its Harnad dual $(W,B)=\HD(V,A)$ are 
written as
\[
A(z)=Q(z\,\unit_W -T)^{-1}P, \quad
B(\zeta)=T+\frac{PQ}{\zeta}.
\]
Consider the operator 
\[
\zeta \Bigl( \frac{d}{d\zeta} - B(\zeta) + \zeta^{-1} \unit_W \Bigr)
= \zeta \frac{d}{d\zeta} -T\zeta -PQ +\unit_W
\] 
corresponding to the system $B(\zeta)-\zeta^{-1}\,\unit_W$.
The inverse Fourier-Laplace transform 
$\mathfrak{F}^{-1} \colon \zeta \mapsto -d/dz,\, d/d\zeta \mapsto z$ 
of it is given by
\begin{align*}
-\frac{d}{dz} z + T \frac{d}{dz} -PQ +\unit_W
&= -\Bigl( \unit_W + z \frac{d}{dz} \Bigr) + T \frac{d}{dz} -PQ +\unit_W \\
&= -(z\,\unit_W - T)\frac{d}{dz} - PQ,
\end{align*}
which corresponds to the system $\widetilde{A}(z):=-(z\,\unit_W -T)^{-1}PQ$.
The relation
\[
A(z)Q=Q(z\,\unit_W -T)^{-1}PQ = -Q \widetilde{A}(z)
\]
means that $Q \colon W \to V$ 
gives a morphism from $(W,-\tilde{A})$ to $(V,A)$.
Note that $Q$ is surjective and $P$ is injective because 
the pairs $(\range Q, W)$ and $(\Ker P,0)$ are both subrepresentations 
of the irreducible Harnad datum $(V,W,0,T,Q,P)$.
In particular, $(V,A)$ is an irreducible quotient of $(W,-\tilde{A})$. 
So we have a commutative diagram
\[
\begin{CD}
(V,A) @>{\HD}>> (W,B) \\
@A{\text{quotient}}AA  @VV{\text{shift by $-\zeta^{-1}$}}V \\
(W,-\widetilde{A}) @<{-\mathfrak{F}^{-1}}<< (W,B-\zeta^{-1}).
\end{CD}
\]

When $T$ is semisimple, a regular singular system of the form
\[
(z\,\unit_W -T) \frac{du}{dz} = R u, \quad R \in \End(W)
\]
is called a system of {\em Okubo normal form}  
and has been studied by many researchers 
(see e.g.\ \cite{BJL1,Okubo,HY,Yok2}).
Kawakami~\cite{Kawakami} further studied systems of the above form 
for arbitrary $T$ (then irregular singularities appear), 
and construct a functor from the category of triples $(W,T,R)$ to that of 
pairs $(V,A)$, which is defined as follows:
for given triple $(W,T,R)$, 
write $R=PQ$, where 
$Q \colon W \to V:=W/\Ker R$ is the projection and 
$P \colon V \to W$ is the injection induced from $R$,
and then send it to the pair $\bigl( V,Q(z\,\unit_W-T)^{-1}P \bigr)$.
Note that it coincides with 
the left vertical arrow in the previous diagram.
\end{remark}

\begin{remark}\label{rem:globaldual}
So far we have restricted our Harnad data  
by Conditions \eqref{eq:T} and \eqref{eq:S} 
using some fixed $D,E,\vek,\vel$.
Considering here various $D,E,\vek,\vel$ at once, 
we obtain the category $\mathcal{H}$ of all Harnad data,
and similarly, the category $\mathcal{D}$ of 
pairs $(V,A)$ consisting of a finite-dimensional $\C$-vector space $V$ 
and an $\End(V)$-valued rational function $A(z)$ 
which is bounded locally near $\infty$.
Then using Remark~\ref{rem:canonical},~(b), 
we can show that 
the functors $\Phi, \kappa, \HD$ are uniquely extended to 
functors $\Phi \colon \mathcal{H} \to \mathcal{D},\, 
\kappa \colon \mathcal{D} \to \mathcal{H},\, 
\HD \colon \mathcal{D} \to \mathcal{D}$.
All the results, except Theorem~\ref{thm:duality-geom}, 
obtained in this section can be rephrased 
in terms of these categories/functors.
\end{remark}

\section{Generalized middle convolution}\label{sec:mc}

For $\alpha \in \EE_\vel(\C)$, we define 
the {\em addition} functor with $\alpha$ by
\[
\add_\alpha \colon \kDl \to \kDl; \quad (W,B) \mapsto (W,B+\alpha).
\]
Now the generalized middle convolution is defined as follows:
\begin{definition}\label{dfn:extmc}
For $\alpha \in \EE_\vel(\C)$, we define
\[
\mc_\alpha := \HD \circ \add_\alpha \circ \HD \colon \lDk \to \lDk,
\]
which we call the {\em middle convolution} functor with $\alpha$.
\end{definition}

\begin{example}\label{ex:mc}
(a) If $0 \in E,\, \alpha(\zeta)=\lambda/\zeta$ 
and $(V,A)$ is Fuchsian, 
$\mc_\alpha(V,A)$ coincides with 
the original middle convolution $\mc_\lambda(V,A)$.

(b) As a simple example, let us consider the case that
$D,E =\{ 0 \}$ and $(V,A)=(\C,\alpha)$ is of rank 1. 
We omit the subscript $t=0$ as usual. 
By virtue of Example~\ref{ex:dual},~(b), 
we already know that in this case
the Harnad dual $(W,B)=\HD(\C,\alpha)$ is given by
\[
W=\C^d, \quad B(\zeta)=N+\frac{R}{\zeta},
\]
where $d$ is the pole order of $\alpha$, 
$N$ is the $d \times d$ nilpotent single Jordan block, 
and $R$ is defined by
\[
R = 
\left(\,\begin{matrix}
                        0  & 0              & \cdots &  0 \\
                    \vdots & \vdots         &        &  \vdots \\
                        0  & 0              & \cdots &  0   \\
            \alpha_d & \alpha_{d-1} & \cdots & \alpha_1 
\end{matrix} \,\right) \in \End(\C^d).
\]
Now let us compute 
the middle convolution 
$\mc_{\lambda/\zeta}(\C,\alpha)=\HD(\C^d,B+\lambda/\zeta)$ 
with $\lambda/\zeta \neq 0$.
If $\lambda + \alpha_1 \neq 0$, 
the matrix $R+\lambda\,\unit_{\C^d}$ is invertible.
Hence the canonical datum $(\C^d,V',0,P',Q')$ for 
$\zeta^{-1}(R+\lambda\,\unit_{\C^d})$ 
is given by
\[
V'=\C^d, \quad P'=R+\lambda\,\unit_{\C^d}, \quad Q'=\unit_{\C^d},
\]
whence
\[
\mc_{\lambda/\zeta}(\C,\alpha) =\bigl( \C^d, Q'(z\,\unit_{\C^d} - N)^{-1} P' \bigr)
= \bigl( \C^d, (z\,\unit_{\C^d} - N)^{-1} (R+\lambda\,\unit_{\C^d}) \bigr).
\]
As a matrix, 
$(z\,\unit_{\C^d} - N)^{-1} (R+\lambda\,\unit_{\C^d})$ 
has the $(i,j)$-entry given by
\begin{equation}\label{eq:mc-example}
\bigl[ (z\,\unit_{\C^d} - N)^{-1} (R+\lambda\,\unit_{\C^d}) \bigr]_{i,j}=
\left\{\,
\begin{aligned}
&\frac{\alpha_{d-j+1}}{z^{d-i+1}}+\frac{\lambda}{z^{j-i+1}} & i \leq j, \\
&\frac{\alpha_{d-j+1}}{z^{d-i+1}} & i > j.
\end{aligned}\right.
\end{equation}
If $\lambda +\alpha_1 = 0$, 
the subspace $\Ker (R+\lambda\,\unit_{\C^d})$ is generated by 
the $d$-th coordinate vector in $\C^d$ and hence 
the projection $\C^d \to \C^{d-1}$ killing the $d$-th component
gives an isomorphism $V'=\C^d/\Ker (R+\lambda\,\unit_{\C^d}) \simeq \C^{d-1}$. 
Under this identification, the decomposition
\[
R+\lambda\,\unit_{\C^d} = 
\left(\,\begin{matrix}
             \lambda &          &  0 \\
                     & \ddots   &     \\
                  0  &          &  \lambda   \\
            \alpha_d &  \cdots  & \alpha_2 
\end{matrix} \,\right) 
\left(\,\begin{matrix}
1 & 0 & \cdots & 0 \\
  & \ddots & \ddots & \vdots  \\
0 &        & 1 & 0
\end{matrix}\,\right)
\]
gives the matrices $P'$ and $Q'$.
The middle convolution is given by
\[
\mc_{\lambda/\zeta}(\C,\alpha) 
=\bigl( \C^{d-1}, Q'(z\,\unit_{\C^d} - N)^{-1} P' \bigr).
\]
As in Remark~\ref{rem:laplace}, we have
\[
\bigl[ Q'(z\,\unit_{\C^d} - N)^{-1} P' \bigr] Q' 
= Q' \bigl[ (z\,\unit_{\C^d} - N)^{-1} P'Q' \bigr].
\]
Since $Q'=( \unit_{\C^{d-1}}\, 0 )$, we see that
the $(i,j)$-entry of $Q'(z\,\unit_{\C^d} - N)^{-1} P'$ 
is the same as that of $(z\,\unit_{\C^d} - N)^{-1} P'Q'$, 
which is given by \eqref{eq:mc-example}.
\end{example}

We give three basic properties of the middle convolution.
First rephrasing Theorem~\ref{thm:duality-cat}, we have the following:
\begin{corollary}\label{cor:mc-property1}
If $(V,A) \in \lDk$ is irreducible and 
$(V,A) \not\sim (\C,s)$ for any $s \in E$, 
then $\mc_0(V,A) \sim (V,A)$.
\end{corollary}

Furthermore, Theorem~\ref{thm:duality-cat} 
also implies the following:
\begin{corollary}\label{cor:mc-property2}
Suppose that a pair $(V,A) \in \lDk$ 
and $\alpha \in \EE_\vel(\C)$ satisfy the following conditions:
\begin{itm}
\item[{\rm (a)}] $(V,A)$ is irreducible;
\item[{\rm (b)}] $(V,A) \not\sim (\C,s)$ for any $s \in E$;
\item[{\rm (c)}] $\mc_\alpha(V,A) \neq (0,0)$.
\end{itm}
Then $\mc_\alpha(V,A)$ is also irreducible and
\[
\mc_\beta \circ \mc_\alpha (V,A) 
\sim \mc_{\alpha +\beta}(V,A)
\]
for any $\beta \in \EE_\vel(\C)$.
\end{corollary}

\begin{proof}
Clearly the addition functor $\add_\alpha$ preserves the irreducibility.
Therefore Theorem~\ref{thm:duality-cat} implies that
$\add_\alpha \circ \HD(V,A)$ is irreducible, 
and further that $\mc_\alpha(V,A)$ is irreducible 
if $\add_\alpha \circ \HD(V,A) \not\sim (\C,t)$ 
for any $t \in D$, or equivalently (by Corollary~\ref{cor:exceptional}), if
$\mc_\alpha(V,A) \neq (0,0)$.
Furthermore, under the same assumption we have
\begin{align*}
\mc_\beta \circ \mc_\alpha (V,A) 
&=\HD \circ \add_\beta \circ (\HD \circ \HD) 
\circ \add_\alpha \circ \HD (V,A) \\
&\sim \HD \circ \add_\beta \circ \add_\alpha \circ \HD (V,A) \\
&=\HD \circ \add_{\alpha+\beta} \circ \HD (V,A)
=\mc_{\alpha + \beta}(V,A).
\end{align*}
\end{proof}

Theorem~\ref{thm:duality-geom} and Remark~\ref{rem:moduli} 
imply the following (the proof is immediate):

\begin{corollary}\label{cor:mc-isom}
Suppose that a pair $(V,A) \in \lDk$ 
and $\alpha \in \EE_\vel(\C)$ satisfy 
all the assumptions in Corollary~\ref{cor:mc-property2}.
Let $\bO \times \bO_\infty \subset \EE_\vek(V) \times \g^*_2(V)$ 
be the truncated formal type of $A(z)$, namely, the
$G_\vek(V) \times G_2(V)$-coadjoint orbit
such that $A(z)$ gives a point in $\Mreg(\bO,\bO_\infty)$.
Similarly, let $(V^\alpha,A^\alpha) := \mc_\alpha(V,A)$ and let
$\bO^\alpha \times \bO^\alpha_\infty \subset 
\EE_\vek(V^\alpha) \times \g^*_2(V^\alpha)$ 
be the truncated formal type of $A^\alpha(z)$.
Then the middle convolution induces a symplectomorphism
\[
\mc_\alpha \colon \Mreg(\bO,\bO_\infty) 
\xrightarrow{\simeq} 
\Mreg(\bO^\alpha,\bO^\alpha_\infty).
\]
\end{corollary}

The rest of this section is devoted to 
study the behavior of truncated formal type 
$(\bO,\bO_\infty) \mapsto (\bO^\alpha,\bO^\alpha_\infty)$
under the middle convolution $\mc_\alpha$.
The following example treats 
the behavior of $\bO_\infty$ in some special, but important case.

\begin{example}\label{ex:ref}
Suppose that $\bO_\infty = w^{-1}\lambda\,\unit_V$ 
for some $\lambda \neq 0$ 
(where $w=z^{-1}$ is the indeterminate for $\bO_\infty$ 
as in Remark~\ref{rem:moduli}). 
Then for any representative $A(z)$ of a point in 
$\Mreg(\bO,\bO_\infty)$, we have
\[
S=0, \quad \Res_{z=\infty} A(z) = \lambda\,\unit_V.
\]
Let $(V,W,T,Q,P)$ be the canonical datum for $A(z)$. 
Since $QP=-\lambda\,\unit_V$ and $\lambda \neq 0$, 
we see that $Q$ is surjective and $P$ is injective. 
Set $V'=\Coker P$ and let $Q' \colon W \to V'$ 
be the natural projection. Then we have an exact sequence
\[
\begin{CD}
0 @>>> V @>{P}>> W @>{Q'}>> V' @>>> 0,
\end{CD}
\]
and $-\lambda^{-1}Q$ gives a splitting of it. 
Let $P' \colon V' \to W$ be the injection such that 
$\lambda^{-1}P'$ is the homomorphism associated to this splitting;
\[
P (-\lambda^{-1}Q) + \lambda^{-1}P'Q'=\unit_W,
\]
namely,
\[
P'Q' -PQ = \lambda\,\unit_W.
\]
Then one can easily see that $\range P=\Ker (PQ +\lambda\,\unit_W)$ 
and $Q'$ coincides with 
the homomorphism induced from $PQ+\lambda\,\unit_W$.
This observation shows that the datum $(V',W,T,Q',P')$ 
coincides with the canonical datum for  
$\mc_{\lambda/\zeta}(V,A)$. 
By the definition we have $Q'P'=\lambda\,\unit_{V'}$, 
and hence 
$\bO^{\lambda/\zeta}_\infty = -w^{-1}\lambda\,\unit_{V^{\lambda/\zeta}}$.
Thus we see that the middle convolution $\mc_{\lambda/\zeta}$ 
induces a symplectomorphism 
\[
\Mreg (\bO, w^{-1}\lambda\,\unit_V) 
\xrightarrow{\simeq} 
\Mreg (\bO^{\lambda/\zeta}, 
-w^{-1}\lambda\,\unit_{V^{\lambda/\zeta}}),
\]
and $\dim V^{\lambda/\zeta} = \dim W - \dim V$.
We will use it in the next section.
\end{example}

From now on, we mainly treat systems satisfying the following property:
\begin{definition}\label{dfn:HTL}
An element in $\EE_k(V)$ of the form
\[
\Lambda(z)=\bigoplus_{\lambda \in \Sigma} 
\Bigl( \lambda(z)\,\unit_{V_\lambda} + \frac{\Gamma_\lambda}{z} \Bigr),
\]
associated to
\begin{itm}
\item a finite subset $\Sigma \subset \EE_k(\C)$ 
whose elements are all residue-free;
\item a decomposition $V=\bigoplus_{\lambda \in \Sigma} V_\lambda$
by nonzero subspaces $V_\lambda,\, \lambda \in \Sigma$; and
\item matrices $\Gamma_\lambda \in \End(V_\lambda),\, \lambda \in \Sigma$,
\end{itm}
is called a {\em Hukuhara-Turrittin-Levelt normal form}, 
or simply, a normal form. 

An element $A(z) \in \EE_k(V)$ is said to {\em have a normal form} if 
it is equivalent to some normal form under the $G_k(V)$-action. 
A normal form $\Lambda(z)$ equivalent to $A(z)$ 
(which is unique up to $\GL(V)$-conjugation 
by Proposition~\ref{prop:HTL-unique} below)
is called the {\em normal form of $A(z)$},
and each $\lambda \in \Sigma$ is called a {\em spectrum} of $A(z)$.

A system $A(z) \in \EE_\vek(V)$, or a pair $(V,A) \in \lDk$, 
is said to {\em have a normal form at $t \in D$}
if its $\EE_{k_t}(V)$-component $A_t(z)$ has a normal form.
\end{definition}

Note that for any normal form $\Lambda(z)$, 
its coefficient matrices  
\[
\Lambda_i \in \End(V),\quad \Lambda(z)=\sum_{i=1}^k \Lambda_i z^{-i},
\]
satisfy the following properties:
\begin{enm}
\item[(a)] $\Lambda_1, \dots, \Lambda_k$ commute with one another;
\item[(b)] $\Lambda_2, \dots, \Lambda_k$ are all semisimple.
\end{enm}
Each simultaneous eigenvalues $(\lambda_2, \dots ,\lambda_k)$
of $(\Lambda_2,\dots ,\Lambda_k)$ give a spectrum of $\Lambda(z)$ 
by $\lambda(z)=\sum_{i=2}^k \lambda_i z^{-i}$, 
and the subspace $V_\lambda$ is given by 
the corresponding simultaneous eigenspace.

\begin{proposition}\label{prop:HTL-unique}
If two normal forms $\Lambda(z), \Lambda'(z)$ are contained in 
the same $G_k(V)$-coadjoint orbit,
then there exists $a \in \GL(V)$ such that 
$\Lambda'(z)=a \Lambda(z) a^{-1}$.
\end{proposition}

\begin{proof}
Let $\Lambda(z)=\sum \Lambda_i z^{-i},\,\Lambda'(z)=\sum \Lambda'_i z^{-i}$
and suppose that there exists $g(z)=\sum_i g_i z^i \in G_k(V)$ 
such that $g \cdot \Lambda = \Lambda'$.
It means $g(z)\Lambda(z)z^k = \Lambda'(z)g(z)z^k$ modulo $z^k$, 
which can be written as
\[
\sum_{i=0}^{k-l} \bigl( \Lambda_{l+i}g_i-g_i \Lambda'_{l+i} \bigr) 
=0, \quad l=1,2, \dots ,k.
\]
Looking at the relation for $l=k$, we obtain $g_0 \Lambda_k = \Lambda' g_0$.
So taking the conjugation by $g_0$ on both sides of the relation 
$g \cdot \Lambda = \Lambda'$, we may assume that 
$g_0=1$ and $\Lambda_k=\Lambda'_k$.
Set $\frh_{k+1}:= \gl(V)$ and inductively
\[
\frh_l := \Ker \bigl( \ad_{\Lambda_l}|_{\frh_{l+1}} \bigr), \quad
\frh'_l := \range \bigl( \ad_{\Lambda_l}|_{\frh_{l+1}} \bigr),\quad
l=1,2, \dots ,k.
\]
Note that $\frh'_l \subset \frh_{l+1}$ for $l \geq 1$ as 
$\Lambda_1, \dots ,\Lambda_k$ commute with one another.
Moreover, since $\Lambda_2, \dots ,\Lambda_k$ are all semisimple, we have
\[
\frh_{l+1} = \frh_l \oplus \frh'_l, \quad l=2,\dots ,k.
\]
Now looking at the relation for $l=k-1$, we have
\[
[\Lambda_k,g_1] \in \frh'_k, \quad \Lambda_{k-1}, \Lambda'_{k-1} \in \frh_k, 
\quad [\Lambda_k,g_1] + \Lambda_{k-1} = \Lambda'_{k-1}.
\]
Since $\gl(V)=\frh'_k \oplus \frh_k$, we obtain 
\[
[\Lambda_k,g_1]=0, \quad \Lambda_{k-1}= \Lambda'_{k-1}.
\]
The first relation means $g_1 \in \frh_k$.
Next looking at the relation for $l=k-2$, we obtain
\[
[\Lambda_k,g_2] \in \frh'_k, \quad [\Lambda_{k-1},g_1] \in \frh'_{k-1}, 
\quad [\Lambda_k,g_2] + [\Lambda_{k-1},g_1] + \Lambda_{k-2} = \Lambda'_{k-2}.
\]
Since $\gl(V)=\frh'_k \oplus \frh'_{k-1} \oplus \frh_{k-1}$ and 
$\Lambda_{k-2},\Lambda'_{k-2} \in \frh_{k-1}$, 
we obtain
\[
[\Lambda_k,g_2]=0, \quad [\Lambda_{k-1},g_1]=0, 
\quad \Lambda_{k-2} = \Lambda'_{k-2}.
\]
The first two relations mean $g_2 \in \frh_k,\, g_1 \in \frh_{k-1}$. 

Iterating the argument inductively, we finally obtain 
\[
g_i \in \frh_{i+1}, \quad
\Lambda_i = \Lambda'_i, \quad i=1, \dots ,k-1
\]
from the decomposition 
$\gl(V)=\frh'_k \oplus \frh'_{k-1} \oplus \cdots \oplus \frh'_2 \oplus \frh_2$.
\end{proof}

\begin{remark}\label{rem:HTL-semisimple}
If the leading term of $A(z)=\sum_i A_i z^{-i} \in \EE_k(V)$ is 
diagonalizable with distinct eigenvalues, 
then $A(z)$ has a normal form $\Lambda(z)=\sum_i \Lambda_i z^{-i}$,
where $\ord(\Lambda)=\ord(A) (=: d)$ and 
$\Lambda_i$ is given by the $\Ker (\ad_{A_d})$-part of $A_i$ 
relative to the decomposition 
$\gl(V)=\Ker (\ad_{A_d}) \oplus \range (\ad_{A_d})$ 
(i.e., the `diagonal part' of $A_i$).

Conversely, one can easily show that if $A(z) \in \EE_k(V)$ has a normal form,
then its leading term is semisimple.
\end{remark}

\begin{remark}\label{rem:HTL-infty}
Of course we can also define a similar notion for the singularity at $\infty$;
$(V,A) \in \lDk$ is said to have a normal form at $\infty$ if
the principal part 
\[
A_\infty(w)=-Sw^{-2}+w^{-1}\Res_{z=\infty} A(z) \in \EE_2(V) 
\]
of the Laurent expansion of $A(z)dz$ at $\infty$ has a normal form.
Here one can observe that in general,
an element $Sw^{-2} + Rw^{-1} \in \EE_2(V)$ has a normal form 
if and only if $S$ is semisimple.
Hence $(V,A) \in \lDk$ has a normal form at $\infty$ if and only if 
$S=\lim_{z\to \infty}A(z)$ is semisimple, 
which is equivalent to that $\HD(V,A)$ is Fuchsian at any $s \in E$ 
if $(V,A)$ is irreducible.
\end{remark}

\begin{remark}\label{rem:unramified}
Let $\FF := \C((z))$ be the field of formal Laurent series and 
consider each $A(z) \in \EE_k(V)$ as an element of 
$\End_{\FF} (V \otimes \FF) \simeq \gl(V) \otimes \FF$, 
on which the group $\Aut_{\FF} (V \otimes \FF)$ 
acts as gauge transformations;
\[
A \mapsto g [A] := gAg^{-1} + \frac{dg}{dz}g^{-1}, 
\quad g \in \Aut_{\FF} (V \otimes \FF).
\]
A fundamental fact in the Hukuhara-Turrittin-Levelt theory 
of meromorphic linear ordinary differential equations
~\cite{Huk,Tur,Levelt} is that 
for any system $A(z) \in \gl(V) \otimes \FF$, 
there exists a positive integer $b$ such that 
the pull back of the form $A(z)dz$ via the ramified covering map 
$f_b \colon z \mapsto z^b$ is 
equivalent to some normal form $\Lambda(z)dz$ 
under the $\Aut_{\FF}(V \otimes \FF)$-action.
For this reason, a system $A(z)$ is said to be {\em unramified} 
if it is equivalent to some normal form $\Lambda(z)$ under the 
$\Aut_{\FF}(V \otimes \FF)$-action (i.e., in the case of $b=1$).

It is also known (see \cite[\S 6--7]{BV-group}) that 
two normal forms 
$\Lambda(z)=\sum \Lambda_i z^{-i},\, 
\Lambda'(z)=\sum \Lambda'_i z^{-i}$ are equivalent 
under the $\Aut_{\FF}(V \otimes \FF)$-action if and only if
those have the same pole order and there exists $a \in \GL(V)$ 
such that 
\[
a\Lambda_ia^{-1}=\Lambda'_i,\; i \geq 2, \quad
a\exp \bigl( 2\pi \sqrt{-1} \Lambda_1 \bigr) a^{-1}
= \exp \bigl( 2\pi \sqrt{-1} \Lambda'_1 \bigr).
\]
It implies that for any unramified system $A(z) \in \gl(V) \otimes \FF$,
there exists a {\em unique} normal form $\Lambda(z)$ 
equivalent to $A(z)$ such that
the real parts of all the eigenvalues 
of the residue $\Lambda_1=\Res_{z=0}\Lambda(z)$ are in $[0,1)$.
A normal form satisfying such a condition is said to be {\em reduced}.

Note that in the above we consider  
formal {\em meromorphic} gauge transformations.
Replacing the action of $\Aut_{\FF}(V \otimes \FF)$ with that of the subgroup 
$\Aut_{\C[[z]]}(V \otimes \C[[z]])$ 
(in other words, considering only formal gauge transformations)
in the above definition of unramifiedness, 
we obtain a notion close to 
the one in Definition~\ref{dfn:HTL};
in fact, the following two conditions 
for a system $A(z) \in \EE_k(V)$ are equivalent:
\begin{enm}
\item[(a)] $A(z)$ is equivalent to a reduced normal form 
under the $G_k(V)$-action; 
\item[(b)] $A(z)$ is equivalent to a reduced normal form
under the $\Aut_{\C[[z]]}(V \otimes \C[[z]])$-action.
\end{enm}
The direction (b) $\Rightarrow$ (a) is clear because 
the relation $g[A]=\Lambda$ induces the relation 
$\ov{g} \cdot A=\Lambda$ in $\EE_k(V)$, where $\ov{g}(z) \in G_k(V)$ 
is the element induced from $g$ by taking modulo $z^k$ 
(so this direction is also true in the non-reduced case).
The proof of the direction (a) $\Rightarrow$ (b) is as follows. 
For given relation $g \cdot A =\Lambda$ in $\EE_k(V)$, 
consider $g$ as an element of $\Aut_{\C[[z]]}(V \otimes \C[[z]])$. 
Then it is easy to see that $g[A]-\Lambda \in \gl(V) \otimes \C[[z]]$.
Now use \cite[Lemma 1.4.1]{BV-ast}.

Note that Condition~(b) is not equivalent to the unramifiedness.
For instance, let $\Gamma$ be the $2 \times 2$ nilpotent single Jordan block 
and set $\Lambda(z):=\Gamma z^{-1}$ 
which is a normal form by definition.
Take $g(z):=\operatorname{diag}(1,z^k)$ for some integer $k>1$. 
Then the unramified element $A(z):=g[\Lambda]$ has pole order $k$ and 
no $\C[[z]]$-part, so it is contained in $\EE_k(\C^2)$. 
However it is not equivalent to a normal form 
under the $\Aut_{\C[[z]]}(V \otimes \C[[z]])$-action 
because its leading term is nilpotent.
\end{remark}

Recall that by Theorem~\ref{thm:geom}, 
each $G_\vek(V)$-coadjoint orbit 
can be described as 
the symplectic quotient $\mu_T^{-1}(\OO)^{T\st}/G_T$ 
for some vector space $W$, an endomorphism $T \in \End(W)$ and 
a $G_T$-coadjoint orbit $\OO \subset \g_T^*$ via the map $\Phi_T$.
We denote it by $\mu_T^{-1}(\OO)_V^{T\st}/G_T$
when we want to emphasize the vector space $V$.
In the two lemmas below, we assume for simplicity that $D=\{ 0 \}$ 
and omit the subscript $t=0$ as usual.
Also for $T \in \gl(W)$, we denote by 
$p_T \colon \gl(W) \to \g_T^*$ the natural projection onto 
the dual of the centralizer $\g_T$ of $T$ as before.

\begin{lemma}\label{lem:mc-sing}
Let $W$ be a nonzero finite-dimensional $\C$-vector space,
$N \in \End(W)$ be a nilpotent endomorphism with $N^k=0$, 
and $\OO \subset \g_N^*$ be a $G_N$-coadjoint orbit.
Suppose that finite-dimensional $\C$-vector spaces $V, V' \neq 0$ 
and $\alpha \in \C$ satisfy the following conditions:
\begin{enm}
\item[{\rm (a)}] $\Phi_N \bigl( \mu_N^{-1}(\OO)_V^{N\st}/G_N \bigr)$ 
contains a normal form;
\item[{\rm (b)}] 
$\mu_N^{-1}\bigl(\OO+p_N(\alpha\,\unit_{W})\bigr)_{V'}^{N\st}$ 
is non-empty,
\end{enm}
Then the orbit
$\Phi_N \Bigl( \mu_N^{-1}\bigl(\OO+p_N(\alpha\,\unit_{W})
\bigr)_{V'}^{N\st}/G_N \Bigr)$ 
contains a normal form which has the same nonzero spectra 
as that in {\rm (a)}.
\end{lemma}

\begin{proof}
By the assumption, the image of $\mu_N^{-1}(\OO)_V^{N\st}/G_N$ 
under the map $\Phi_N$ contains a normal form
\[
\Lambda(z)=
\bigoplus_{\lambda \in \Sigma} 
\Bigl( \lambda(z)\,\unit_{V_\lambda} + \frac{\Gamma_\lambda}{z} \Bigr), \quad
\lambda(z)=\sum_{i=2}^k \lambda_i z^{-i},\; 
V=\bigoplus_{\lambda \in \Sigma} V_\lambda,\;
\Gamma_\lambda \in \End(V_\lambda).
\]
In what follows we assume $0 \in \Sigma$. 
Replacing $\Sigma$ with $\Sigma \cup \{ 0 \}$ and 
setting $V_0 := 0,\, \Gamma_0 :=0$ make the argument below work well 
also in the case $0 \notin \Sigma$.

Set $d_\lambda := \ord(\lambda + \Gamma_\lambda/z)$.
By Remark~\ref{rem:canonical}, 
we see that the canonical datum for $\Lambda(z)$ is given by
the direct sum of those $(V_\lambda,W_\lambda,N_\lambda,Q_\lambda,P_\lambda)$
for $\lambda(z)\,\unit_{V_\lambda}+\Gamma_\lambda/z$.
The direct summand for $\lambda \neq 0$ is given by
\begin{gather*}
W_\lambda = V_\lambda \otimes \C^{d_\lambda}, \\
Q_\lambda =  \begin{pmatrix} 
\lambda_{d_\lambda}\unit_{V_\lambda}\, & 
\lambda_{d_\lambda-1}\unit_{V_\lambda} & 
\cdots & 
\lambda_2\unit_{V_\lambda}\, &
\Gamma_\lambda \end{pmatrix} \in \Hom(W_\lambda,V_\lambda), \\ 
P_\lambda = \left(\, \begin{matrix} 0 \\ \vdots \\ 0 \\ \unit_{V_\lambda}
       \end{matrix} \,\right) \in \Hom(V_\lambda,W_\lambda), \quad
N_\lambda = \left(\,\begin{matrix}
                        0     & \ \unit_{V_\lambda} \ &        & \ 0 \\
                              & 0            & \ddots &   \\
                              &              & \ddots & \ \unit_{V_\lambda}   \\
                        0     &              &        & \ 0 
\end{matrix} \,\right)
\in \End(W_\lambda),
\end{gather*}
and for $\lambda=0$, given by
\begin{gather*}
W_0 = V_0/\Ker \Gamma_0, \quad N_0=0, \\
Q_0 \colon W_0 \xrightarrow{\Gamma_0} V_0, \quad
P_0 \colon V_0 \xrightarrow{\text{projection}} W_0.
\end{gather*}
By Proposition~\ref{prop:unique}, we may assume that 
$W=\bigoplus_\lambda W_\lambda$ and $N=\bigoplus_\lambda N_\lambda$.
We set $P:=\bigoplus_\lambda P_\lambda,\, Q:=\bigoplus_\lambda Q_\lambda$.

Now take a point 
$(Q',P') \in \mu_N^{-1}\bigl(\OO+p_N(\alpha\,\unit_W)\bigr)_{V'}^{N\st}$.
By using the $G_N$-action, 
we may assume that $\mu_N(Q',P')=\mu_N(Q,P)+p_N(\alpha\,\unit_W)$, i.e.,
\[
P'Q'=PQ -\alpha\,\unit_W + [N,X]
\]
for some $X \in \End(W)$. 
Let us write it as
\begin{align*}
(P'Q')_{\lambda \lambda} &= 
P_\lambda Q_\lambda -\alpha\,\unit_{W_\lambda} 
+ [N_\lambda ,X_{\lambda \lambda} ], \\
(P'Q')_{\lambda \mu}
&= N_\lambda X_{\lambda \mu} 
- X_{\lambda \mu} N_\mu, \quad \lambda \neq \mu,
\end{align*}
where the subscript `$\lambda \mu$' means 
the $\Hom(W_\mu,W_\lambda)$-block component.
By Lemma~\ref{lem:stability}, the restriction $Q'|_{\Ker N}$ is injective.
Regarding $\Ker N_\lambda$ as a subspace of 
$\Ker N \subset W$ using the direct sum decomposition, 
we set $V'_\lambda := Q'(\Ker N_\lambda)$ for $\lambda \neq 0$ and 
\[
V'_0 := V'/\bigoplus_{\lambda \neq 0} V'_\lambda.
\]
Since the composite of $P'Q' |_{\Ker N_\lambda} \colon \Ker N_\lambda \to W$ 
and the projection $\pi_0 \colon W \to W_0$ is
\[
(P'Q')_{0\lambda} |_{\Ker N_\lambda}
= (N_0 X_{0\lambda} -X_{0\lambda} N_\lambda)|_{\Ker N_\lambda}
=0
\]
for $\lambda \neq 0$, 
the map $\pi_0 \circ P' \colon V' \to W_0$ descends to a map 
\[
V'_0 = V'/\bigoplus_{\lambda \neq 0} V'_\lambda \to W_0,
\]
which we denote by $P^\alpha_0$. Note that $\Ker \pi_0 \supset \range N$ 
since $N_0=0$. 
Hence by Lemma~\ref{lem:stability}, $\pi_0 \circ P'$ is surjective
and hence so is $P^\alpha_0$.
Also we define $Q^\alpha_0 \colon W_0 \to V'_0$ by the composite
\[
W_0 \xrightarrow{\text{inclusion}} W \xrightarrow{Q'} V' 
\xrightarrow{\text{projection}} V'_0,
\]
which is injective since $W_0 \subset \Ker N$.
Clearly the pair $(Q^\alpha_0,P^\alpha_0)$ satisfies
\[
P^\alpha_0 Q^\alpha_0 = (P' Q')_{00} = P_0 Q_0 -\alpha\,\unit_{W_0}.
\]
For $\lambda \neq 0$, let
$(V_\lambda,W^\alpha_\lambda,N^\alpha_\lambda,
Q^\alpha_\lambda,P^\alpha_\lambda)$
be the canonical datum for the system
\[
\lambda(z)\,\unit_{V_\lambda} + \frac{\Gamma_\lambda}{z} 
- \frac{d_\lambda \alpha}{z}\,\unit_{V_\lambda}
\in \EE_k(V^\lambda).
\]
By Remark~\ref{rem:canonical} we have 
$(W^\alpha_\lambda,N^\alpha_\lambda)=(W_\lambda,N_\lambda)$ and
\[
Q^\alpha_\lambda = Q_\lambda - 
\begin{pmatrix} 
0\, & \cdots & 0\, &
d_\lambda \alpha\,\unit_{V_\lambda} \end{pmatrix}, \quad
P^\alpha_\lambda =P_\lambda.
\] 
Now identifying $V'_\lambda$ with $V_\lambda$ 
by using the injection $Q'|_{\Ker N_\lambda}$ for $\lambda \neq 0$, set
\[
(V',W,N,Q^\alpha,P^\alpha) :=
(V'_0,W_0,0,Q^\alpha_0,P^\alpha_0) \oplus \bigoplus_{\lambda \neq 0} 
(V_\lambda,W_\lambda,N_\lambda,Q^\alpha_\lambda,P^\alpha_\lambda).
\]
Then it is stable by Remark~\ref{rem:canonical} and satisfies
\begin{equation}\label{eq:mc-normal}
\Phi_N(Q^\alpha,P^\alpha) =  
\frac{Q^\alpha_0 P^\alpha_0}{z} \oplus
\bigoplus_{\lambda \neq 0} 
\Bigl( \lambda(z)\,\unit_{V_\lambda} + \frac{\Gamma_\lambda}{z} 
- \frac{d_\lambda \alpha}{z}\,\unit_{V_\lambda} \Bigr)  
\in \EE_k(V').
\end{equation}
Furthermore, for $\lambda \neq 0$ we have
\begin{align*}
P^\alpha_\lambda Q^\alpha_\lambda - 
P_\lambda Q_\lambda + \alpha\,\unit_{W_\lambda} &= 
\unit_{V_\lambda} \otimes 
\operatorname{diag}
\bigl( \alpha, \cdots , \alpha, (1-d_\lambda)\alpha \bigr) \\
&= [N_\lambda, X_\lambda],
\end{align*}
where 
\[
X_\lambda := \unit_{V_\lambda} \otimes \left(\, \begin{matrix}
0 \quad &  & \cdots &  & 0 \\
\alpha \quad & \ 0 \ &  & &  \\
0 \quad & \ 2\alpha \ & \ddots & & \vdots \\
\vdots \quad &  \ddots  & \ddots & 0 & \\
0 \quad &  \cdots       & \ 0 \ & (d_\lambda-1) \alpha & 0 
\end{matrix}\,\right).
\]
Thus we obtain 
\[
P^\alpha Q^\alpha - PQ + \alpha\,\unit_W = 
[N, \bigoplus X_\lambda] \in \Ker p_N,
\]
which implies
\[
\mu_N(Q^\alpha,P^\alpha) = \mu_N(Q,P) + p_N(\alpha\,\unit_W) 
\in \OO + p_N(\alpha\,\unit_W).
\]
The result follows.
\end{proof}

\begin{remark}\label{rem:mc-residue}
In the above proof, 
fix a collection of complex numbers 
$\xi_0:=0$, $\xi_1, \dots ,\xi_n$ satisfying 
\[
(\Gamma_0-\xi_0\,\unit_{V_0})(\Gamma_0 -\xi_1\,\unit_{V_0}) 
\cdots (\Gamma_0 - \xi_n\,\unit_{V_0}) =0.
\]
Then it is well-known \cite[\S2]{Cra-par} that the numbers  
\[
\rank \bigl[ (\Gamma_0 -\xi_0\,\unit_{V_0}) 
\cdots (\Gamma_0 -\xi_i\,\unit_{V_0}) \bigr], 
\quad i=0,\dots ,n
\]
characterize the conjugacy class of $\Gamma_0$.
Now consider the matrix $P_0 Q_0$, 
which is by definition the endomorphism of 
$W_0=V_0/\Ker \Gamma_0 \simeq \range \Gamma_0$ 
induced from $\Gamma_0$. It satisfies
\begin{align*}
\rank \prod_{j=1}^i (P_0 Q_0 - \xi_j\,\unit_{W_0}) &=
\rank \prod_{j=1}^i (\Gamma_0|_{\range \Gamma_0} 
- \xi_j\,\unit_{\range \Gamma_0}) \\
&=\rank \prod_{j=0}^i (\Gamma_0 - \xi_j\,\unit_{V_0}), \quad
i=1,2,\dots ,n.
\end{align*}
Hence the matrix $P_0^\alpha Q_0^\alpha = P_0 Q_0 -\alpha\,\unit_{W_0}$ 
satisfies 
\[
\rank \prod_{j=1}^i 
(P_0^\alpha Q_0^\alpha + \alpha\,\unit_{W_0} - \xi_j\,\unit_{W_0}) =
\rank \prod_{j=0}^i (\Gamma_0 - \xi_j\,\unit_{V_0}), \quad 
i=1,2,\dots ,n.
\]
Since $P^\alpha_0$ is surjective and $Q^\alpha_0$ is injective, 
the above implies the following equalities for 
the matrix $\Gamma^\alpha_0 := Q_0^\alpha P_0^\alpha$:
\begin{align*}
\rank \Gamma_0^\alpha &= \dim W_0 = \rank \Gamma_0, \\
\rank \Bigl[ \Gamma_0^\alpha \prod_{j=1}^i 
(\Gamma_0^\alpha + \alpha\,\unit_{V'_0} - \xi_j\,\unit_{V'_0}) \Bigr] &=
\rank \prod_{j=0}^i (\Gamma_0 - \xi_j\,\unit_{V_0}), \quad 
i=1,2,\dots ,n,
\end{align*}
which characterize the conjugacy class of $\Gamma_0^\alpha$.
Therefore the normal form given by \eqref{eq:mc-normal} 
can be computed from that contained in 
$\Phi_N \bigl( \mu_N^{-1}(\OO )^{N\st}_V/G_N \bigr)$ and $\alpha,\, V'$.
Note that in the extreme case $\Gamma_0=0$, 
the first equality in the above implies $\Gamma_0^\alpha =0$.
\end{remark}

\begin{proposition}\label{prop:mc-sing}
Suppose that $(V,A) \in \lDk$ and $\alpha \in \EE_\vel(\C)$ 
satisfy all the assumptions in Corollary~\ref{cor:mc-property2},
and additionally 
that $(V,A)$ has a normal form at some $t \in D$.
Then the middle convolution $\mc_\alpha(V,A)$ 
also has a normal form at the same $t$ with the same nonzero spectra. 
\end{proposition}

\begin{proof}
Set $(V,W,S,T,Q,P) :=\kappa(V,A)$ and
\[
(W,V',T,S',P',Q'):=\kappa \circ \add_\alpha \circ \HD (V,A).
\]
By definition, we have $\mc_\alpha(V,A)=(V',S'+Q'(z\,\unit_W -T)^{-1}P')$
and
\[
P(\zeta\,\unit_V -S)^{-1}Q + \alpha(\zeta)\,\unit_W 
= P'(\zeta\,\unit_{V'}-S')^{-1}Q'.
\]
Taking the residue at $\zeta=\infty$ on both sides, we obtain
\[
-PQ + \Res_{\zeta=\infty} \alpha(\zeta)\,\unit_W = -P'Q',
\]
and further taking the projection 
$\gl(W) \to \gl(W_t) \xrightarrow{p_{N_t}} \g_{N_t}^*$ 
on both sides, we obtain
\[
\mu_{N_t}(Q_t, P_t) + 
p_{N_t} \Bigl( \Res_{\zeta=\infty}\alpha(\zeta)\,\unit_{W_t} \Bigr)
=\mu_{N_t}(Q'_t,P'_t),
\]
where $Q_t,P_t,Q'_t,P'_t$ are the block components of 
$Q,P,Q',P'$ relative to the decomposition $W=\bigoplus_t W_t$ as usual.
Set $\beta := \Res_{\zeta=\infty}\alpha(\zeta)$ and 
let $\OO$ be the $G_{N_t}$-coadjoint orbit 
through $\mu_{N_t}(Q_t,P_t)$. 
Then the above relation implies that the set
\[
\mu_{N_t}^{-1}\bigl( \OO + p_{N_t}(\beta\,\unit_W)\bigr)_{V'}^{N_t\st}/G_{N_t}
\]
is non-empty, because $\mc_\alpha(V,A)$ is irreducible and hence 
$(Q'_t,P'_t)$ gives a point in the above by Lemma~\ref{lem:irr-st}.
Since $\OO, V, V'$ and $\beta$ 
satisfy all the conditions in Lemma~\ref{lem:mc-sing},
the image of the above set under $\Phi_{N_t}$ 
contains a normal form which has the same nonzero spectra as that in 
$\Phi_{N_t} \bigl( \mu_{N_t}^{-1}( \OO )_V^{N_t\st}/G_{N_t} \bigr)$.
\end{proof}

\begin{remark}\label{rem:mc-sing-explicit}
In the situation of Proposition~\ref{prop:mc-sing},
the proof of Lemma~\ref{lem:mc-sing}, 
more specifically \eqref{eq:mc-normal}, 
together with Remark~\ref{rem:mc-residue}
enables us to compute the normal form of $\mc_\alpha(V,A)$ 
from that of $(V,A)$ and the rank of $\mc_\alpha(V,A)$. 
\end{remark}

\begin{remark}\label{rem:mc-sing-infty}
Suppose that $(V,A)$ and $\alpha$ satisfy all the assumptions 
in Corollary~\ref{cor:mc-property2}, 
and additionally $(V,A)$ has a normal form at $\infty$ 
(see Remark~\ref{rem:HTL-infty}).
Then it is easy to see that 
$\mc_\alpha(V,A)$ has a normal form at $\infty$ 
if and only if $\alpha$ is Fuchsian.
Moreover, in this case,
the normal form can be computed as follows.
Let us use the same notation here as in the above proof and  
let $V=\bigoplus_{s \in E} V_s,\,V'=\bigoplus_{s \in E} V'_s$ 
be the eigenspace decomposition 
for $S,S'$ respectively.
Then writing $\alpha(\zeta)=\sum_{s \in E} (\zeta -s)^{-1}\alpha_s$, 
we have
\[
P_s Q_s+\alpha_s\,\unit_W = P'_s Q'_s, \quad s \in E,
\]
where $Q_s,P_s,Q'_s,P'_s$ are the block components of $Q,P,Q',P'$ 
relative to the decomposition $V=\bigoplus_s V_s,\,V'=\bigoplus_s V'_s$.
Set $\Omega_s := Q_s P_s,\, \Omega'_s := Q'_s P'_s$.
Then the normal forms of $(V,A)$ and $\mc_\alpha(V,A)$ at $\infty$ 
are described as
\[
\bigoplus_{s \in E}\bigl( -sw^{-2}\,\unit_{V_s} 
- \Omega_s w^{-1} \bigr), \quad 
\bigoplus_{s \in E}\bigl( -sw^{-2}\,\unit_{V_s} 
- \Omega'_s w^{-1} \bigr).
\]
Note that $P_s, P'_s$ are injective and $Q_s, Q'_s$ are surjective 
because the data $(W,V,S,P,Q)$ and $(W,V',S',P',Q')$ are both stable.
Fix $s \in E$ and let $\eta_1, \dots ,\eta_m$ 
be complex numbers satisfying
\[
(\Omega_s -\eta_1 \unit_{V_s}) \cdots (\Omega_s -\eta_m\,\unit_{V_s})=0.
\]
Set $\eta_0 :=0$. 
Then an argument similar to that in Remark~\ref{rem:mc-residue} shows 
\begin{align*}
\dim V_s &= \rank P_s Q_s = \rank (P'_s Q'_s -\alpha_s\,\unit_W ), \\
\rank \prod_{j=1}^i (\Omega_s - \eta_j\,\unit_W )
&= \rank \prod_{j=0}^i 
(P_s Q_s -\eta_j\,\unit_W ) \\
&= \rank \prod_{j=0}^i 
(P_s Q_s -\alpha_s\,\unit_W -\eta_j\,\unit_W ), 
\quad i=1,2, \dots ,m.
\end{align*}
These characterize the conjugacy class of $P'_s Q'_s$.
If none of $\eta'_i :=\alpha_s + \eta_i,\, i=0,\dots ,m$ are zero,
$P'_s Q'_s$ is invertible and hence 
the above equalities characterize the matrix $\Omega'_s$ also.
If there is some $l$ such that $\eta'_l =0$, 
then for $i=0,1, \dots ,m$, we have
\begin{align*}
\rank \prod_{\genfrac{}{}{0pt}{2}{j \neq l,}{0 \leq j \leq i}} 
\bigl( \Omega'_s - \eta'_j\,\unit_{V'_s} \bigr) 
&= 
\rank \Bigl[ \prod_{\genfrac{}{}{0pt}{2}{j \neq l,}{0 \leq j \leq i}}
\bigl( P_s Q_s - \eta'_j\,\unit_W \bigr) |_{\range P_s Q_s} \Bigr] \\
&=
\left\{
\begin{aligned}
&\rank \Bigl[ P_s Q_s \prod_{j=0}^i 
\bigl( P_s Q_s - \eta'_j\,\unit_W \bigr) \Bigr] & i<l, \\
&\rank \prod_{j=0}^i 
\bigl( P_s Q_s - \eta'_j\,\unit_W \bigr) & i \geq l,
\end{aligned}\right.
\end{align*}
which characterizes the conjugacy class of $\Omega'_s$.
\end{remark}

\begin{remark}\label{rem:globalmc}
We can further generalize the middle convolution.
Let us recall the category $\mathcal{D}$ 
introduced in Remark~\ref{rem:globaldual}. 
For any rank 1 object $(\C,\alpha) \in \mathcal{D}$,
we can define the addition $\add_\alpha$, 
and hence the middle convolution $\mc_\alpha=\HD \circ \add_\alpha \circ \HD$ 
with $\alpha$, as endo-functors of $\mathcal{D}$.
However decomposing the parameter as $\alpha(\zeta)=a + \alpha^0(\zeta)$, 
where $a:=\lim_{\zeta =\infty} \alpha(\zeta)$, 
we can easily see that 
the first component $a$ 
just plays the role of a translation of coordinate $\zeta$ 
in the middle convolution;
$\mc_\alpha(V,A)$ is obtained from $\mc_{\alpha^0}(V,A)$ 
by the coordinate change $\zeta \mapsto \zeta -a$.
Therefore only the case
$\alpha(\zeta)=\alpha^0(\zeta) \in \EE_\vel(\C)$ for some $E, \vel$ 
is essential.
\end{remark}

\section{Generalized Katz's algorithm}\label{sec:Katz}

In this section we generalize Katz's algorithm as 
an application of our middle convolution.
Hereafter we assume $E=\{ 0 \}$ and $\vel=(1)$.
Hence an object of the category $\lDk$ is just a pair $(V,A)$ of 
a finite-dimensional $\C$-vector space $V$ and a system $A(z) \in \EE_\vek(V)$.
We denote by $\DDk^0$ the full subcategory of $\lDk$ 
consisting of objects $(V,A)$ satisfying $\Res_{z=\infty} A(z)=0$, 
namely, that $A(z)$ has no singularity at $\infty$.
For $(V,A) \in \DDk^0$, 
we denote by $\bO(A) \subset \EE_\vek(V)$
the $G_\vek(V)$-coadjoint orbit through $A(z)$.
Note that if $(V,A)$ is irreducible, 
then it represents a point in the space $\Mreg(\bO(A),0)$.

\begin{definition}\label{dfn:rigid}
An irreducible pair $(V,A) \in \DDk^0$ is said to be {\em naively rigid} if 
the space $\Mreg(\bO(A),0)$ consists of only one point.
\end{definition}

Note that the dimension of $\Mreg(\bO,0)$ can be computed 
(if it is nonempty) as
\[
\dim \Mreg(\bO,0) = \dim \bO - 2\dim \GL(V) +2,
\]
because $\Mreg(\bO,0)$ is an open subset of the symplectic quotient 
of $\bO$ by the action of $\GL(V)$ which  
acts on irreducible systems 
with trivial stabilizer $\C^\times$ by Schur's lemma.

\begin{remark}\label{rem:rigidity}
The reason why we say `naively rigid', not just `rigid', 
is the following. 
Katz originally said an irreducible local system 
on $X:=\CP^1 \setminus D$ to be rigid 
if it is determined up to isomorphism by 
its local monodromies around the points in $D$.
A natural generalization of Katz's rigidity to the irregular singular case
was introduced by Bloch-Esnault~\cite{BE},
who said an irreducible algebraic connection 
on $X$ to be rigid 
if it is determined up to isomorphism by 
its {\em formal {\rm (}meromorphic{\rm )} type} on $D$.
On the other hand, in the above we consider 
the {\em truncated formal type} $\bO(A)$, 
not the formal meromorphic type.

Nevertheless, in a good situation, there is an implication 
between the two notions.
Suppose that a pair $(V,A) \in \DDk^0$ has a reduced normal form 
$\Lambda_t(z)=\sum \Lambda_{t,i}z^{-i}$
(see Remark~\ref{rem:unramified}) at each $t \in D$.
Then $\Res_{z=\infty} A(z)=0$ implies
\[
\sum_{t \in D} \tr \Lambda_{t,1} = \sum_{t \in D} \tr \Res_{z=t} A(z) =0,
\]
since $\tr \Res_{z=0} \colon \g^*_k(V) \to \C$ 
is $G_k(V)$-invariant.
Let $\alpha_t^{(1)}, \dots ,\alpha_t^{(n)}$ 
be the eigenvalues of $\Lambda_{t,1}$ 
repeated according to multiplicities (so $n=\dim V$).
We assume that any
collection $(I_t)_{t \in D}$ of nonempty proper sub-index sets 
$I_t \subset \{\, 1, \dots ,n \,\}$ of the same cardinality 
satisfies
\[
\sum_{t \in D} \sum_{i \in I_t} \alpha_t^{(i)} \not\in \Z.
\]
This assumption implies that 
the trivial bundle $\VV := X \times V$ together with 
the algebraic connection $\nabla_A:=d-A(z)dz$ is irreducible 
(and hence that the pair $(V,A)$ is irreducible).
Indeed, suppose that a subbundle of $\VV$ preserved 
by $\nabla_A$ is given.
A trivialization of it then gives 
an algebraic connection
$\nabla_B = d-B(z)dz$ on a trivial bundle $\WW=X \times W$
together with an embedding 
$\phi \colon (\WW,\nabla_B) \hookrightarrow (\VV,\nabla_A)$.
Note that by the assumption and Remark~\ref{rem:unramified}, 
the connection $\nabla_A$ is unramified at any $t \in D$.
Therefore an argument similar to that in \cite[Theorem~6.4]{BV-group} 
shows that at any $t \in D$, 
the connection $\nabla_B$ is also unramified and 
its reduced normal form is given by the restriction of $\Lambda_t$ 
on some subspace of $V$ (which has the same dimension as $W$).
Thus we obtain a subset $I_t \subset \{\, 1, \dots , n\}$ 
of cardinality $\dim W$,
which indexes the eigenvalues of the reduced normal form of $\nabla_B$ at $t$
(repeated according to multiplicities).
Now using the `polar decomposition'
\[
\Aut_{\FF}(\C^m \otimes \FF) = \Aut_{\C[[z]]}(\C^m \otimes \C[[z]]) 
\cdot G \cdot \Aut_{\C[[z]]}(\C^m \otimes \C[[z]]), 
\]
where 
\[
G=\{\,\operatorname{diag}(z^{r_1}, \dots ,z^{r_m}) \mid 
r_i \in \Z,\, r_1 \leq \cdots \leq r_m \,\},
\]
we have
\[
0=\tr \Res_{z=\infty} B(z)= -\sum_{t \in D} \tr \Res_{z=t} B(z) 
\in -\sum_{t \in D} \sum_{i \in I_t} \alpha_t^{(i)} + \Z.
\]
The assumption implies $W=0$ or $W=V$. 
Hence $(\VV, \nabla_A)$ is irreducible.

Under those assumptions, we can show that 
if $(\VV, \nabla_A)$ is rigid, 
then the pair $(V,A)$ is naively rigid, as follows.
Let $(V,A') \in \DDk^0$ be an irreducible pair satisfying $\bO(A')=\bO(A)$.
Then by Remark~\ref{rem:unramified},
the $\EE_{k_t}(V)$-component $A'_t$ of $A'$ is equivalent to 
$\Lambda_t$ under the $\Aut_{\C[[z]]}(V \otimes \C[[z]])$-action 
for any $t \in D$.
On the other hand, the above argument shows that 
the algebraic connection $(\VV, \nabla_{A'})$ is irreducible.
Hence by the rigidity, there is an isomorphism 
$(\VV, \nabla_{A'}) \simeq (\VV, \nabla_A)$,
which gives for each $t \in D$ an element of $\Aut_{\FF}(V \otimes \FF)$ 
connecting $A'_t$ and the $\EE_{k_t}(V)$-component $A_t$ of $A$. 
Here, because $A_t$ and $A'_t$ are both equivalent 
to $\Lambda_t$ under the $\Aut_{\C[[z]]}(V \otimes \C[[z]])$-action, 
\cite[Theorem~7.2]{BV-group} implies that it is in fact an element 
of $\Aut_{\C[[z]]}(V \otimes \C[[z]])$.
Hence the isomorphism extends to a holomorphic bundle isomorphism
$\CP^1 \times V \to \CP^1 \times V$, which must be constant
and hence gives $(V,A') \sim (V,A)$.
\end{remark}

Recall that any element $A(z) \in \EE_k(V)$ defines 
a matrix $\widehat{A} \in \End(V^{\oplus k})$ by \eqref{eq:matrix}.
Then the key lemma is the following:
\begin{lemma}\label{lem:Katz}
For any element $A(z) \in \EE_k(V)$ having a normal form,
there exists $\alpha (z) \in \EE_k(\C)$ such that
\[
\dim Z(A) \leq (\dim V) 
\dim \Ker (\widehat{A} - \widehat{\alpha\,\unit_V}), 
\]
where $Z(A)$ denotes the stabilizer of $A(z)$ in $G_k(V)$.
\end{lemma}

\begin{proof}
Since $\dim Z(A)=\dim Z(g\cdot A)$ for any $g(z) \in G_k(V)$, 
we may assume that $A(z)$ itself is a normal form;
\[
A(z)=\bigoplus_{\lambda \in \Sigma} 
\Bigl( \lambda(z)\,\unit_{V_\lambda} + \frac{\Gamma_\lambda}{z} \Bigr)
=\sum_{i=1}^k \Lambda_i z^{-i}.
\]
Then a slight modification of the argument in the proof of 
Proposition~\ref{prop:HTL-unique} shows
that an element $g(z) =\sum_i g_iz^i \in G_k(V)$ stabilizes
$A(z)$ if and only if
\[
g_{i-1} \in \frh_i:=\Ker (\ad_{\Lambda_i}) 
\cap \cdots \cap \Ker (\ad_{\Lambda_k}), 
\quad i=1,2,\dots ,k.
\]
Hence $\dim Z(A) = \sum_{i=1}^k \dim \frh_i$.
Similarly, one can easily see that 
$v_i \in V,\, i=1,\dots ,k$ satisfy the equation
\[
\left(\, \begin{matrix}
\Lambda_k & \Lambda_{k-1} & \cdots & \quad\Lambda_1 \\
          & \Lambda_k     & \ddots & \quad \vdots \\
          &       & \ddots & \quad\Lambda_{k-1} \\
0         &       &        & \quad\Lambda_k \end{matrix} \,\right) 
\left(\, \begin{matrix}
v_k \\ v_{k-1} \\ \vdots \\ v_1 \end{matrix} \,\right) =0
\]
if and only if
\[
v_i \in \Ker \Lambda_i \cap \dots \cap \Ker \Lambda_k,\, i=1,2,\dots ,k.
\]
Let $V(\alpha_i,\dots ,\alpha_k)$ be the simultaneous eigenspace 
for $(\Lambda_i,\dots ,\Lambda_k)$ with eigenvalue $(\alpha_i,\dots ,\alpha_k)$.
Replacing $\Lambda$ with $\Lambda -\alpha\,\unit_V$ in the above equation, 
we then obtain
\begin{equation}\label{eq:eigenspace}
\dim \Ker (\widehat{A} - \widehat{\alpha\,\unit_V})
= \sum_{i=1}^k \dim V(\alpha_i,\dots ,\alpha_k).
\end{equation}
On the other hand, we have
\[
\frh_i = \bigoplus_{(\alpha_i,\dots ,\alpha_k)}
\gl \bigl( V(\alpha_i,\dots ,\alpha_k) \bigr), \quad i \geq 2.
\]
To describe $\dim \frh_1$ analogously, 
we apply the fact mentioned in Remark~\ref{rem:oshima-normal} to 
the matrix $\Gamma_\lambda$ for each 
$\lambda=\sum_{i \geq 2} \lambda_i z^{-i} \in \Sigma$.
For any eigenvalue $\lambda_1 \in \C$ of $\Gamma_\lambda$, 
there exist a decomposition of the corresponding generalized eigenspace 
\[
\Ker ( \Gamma_\lambda-\lambda_1\,\unit_{V_\lambda} )^{\dim V_\lambda}
= \bigoplus_{j=1}^{r_{\lambda_1,\lambda_2,\dots ,\lambda_k}}
V^j(\lambda_1,\lambda_2,\dots ,\lambda_k)
\]
with $V^1(\lambda_1,\lambda_2,\dots ,\lambda_k)
=V(\lambda_1,\lambda_2,\dots ,\lambda_k)$ and injections 
\[
\iota \colon V^j(\lambda_1,\lambda_2,\dots ,\lambda_k) \hookrightarrow V^{j-1}(\lambda_1,\lambda_2,\dots ,\lambda_k), \quad
j=2, \dots ,r_{\lambda_1,\dots ,\lambda_k},
\]
such that
the matrix $\Gamma_\lambda-\lambda_1\,\unit_{V_\lambda}$ 
restricted on 
$\Ker ( \Gamma_\lambda-\lambda_1\,\unit_{V_\lambda} )^{\dim V_\lambda}$ 
is written as \eqref{eq:oshima-normal}.
Then we can compute~\cite{Oshima-conn} the dimension of the centralizer $Z(\Gamma_\lambda)$ as
\[
\dim Z(\Gamma_\lambda) = \sum_{\lambda_1} 
\sum_{j=1}^{r_{\lambda_1,\lambda_2,\dots ,\lambda_k}} 
\bigl( \dim V^j(\lambda_1,\lambda_2,\dots ,\lambda_k) \bigr)^2,
\]
whence
\[
\dim \frh_1 = \sum_{\lambda \in \Sigma} 
\dim Z(\Gamma_\lambda) = \sum_{(\lambda_1,\dots ,\lambda_k)} 
\sum_{j=1}^{r_{\lambda_1,\lambda_2,\dots ,\lambda_k}} 
\bigl( \dim V^j(\lambda_1,\lambda_2,\dots ,\lambda_k) \bigr)^2.
\]
Thus writing $V^j(\alpha_1, \dots ,\alpha_k) \equiv V^j(\alpha)$ with 
$\alpha=\sum_{i=1}^k \alpha_i z^{-i}$, we obtain
\begin{equation}\label{eq:centralizer}
\dim Z(A) = \sum_\alpha \sum_j \bigl( \dim V^j(\alpha) \bigr)^2 
+\sum_{i=2}^k \sum_{(\alpha_i,\dots ,\alpha_k)} 
\bigl( \dim V(\alpha_i,\dots ,\alpha_k) \bigr)^2.
\end{equation}

Now we prove the lemma. First we rewrite \eqref{eq:eigenspace} as
\begin{align*}
\dim \Ker (\widehat{A} - \widehat{\alpha\,\unit_V})
&= \dim V(\alpha) +\sum_{i=2}^k \sum_{(\beta_1,\dots ,\beta_{i-1})} \sum_j 
\dim V^j(\beta_1,\dots ,\beta_{i-1},\alpha_i,\dots ,\alpha_k) \\
&= \dim V(\alpha) +\sum_{i=2}^k \sum_\beta \sum_j 
\delta_{\alpha_i,\beta_i} \cdots \delta_{\alpha_k,\beta_k}
\dim V^j(\beta),
\end{align*}
where $\delta_{\alpha_i,\beta_i}$ denotes Kronecker's delta symbol.
Setting 
\[
\delta(\alpha,\beta;j) := \delta_{j,1}\delta_{\alpha,\beta} + \sum_{i=2}^k 
\delta_{\alpha_i,\beta_i} \cdots \delta_{\alpha_k,\beta_k}, \quad
\alpha=\sum_{i=1}^k \alpha_iz^{-i},\, 
\beta=\sum_{i=1}^k \beta_i z^{-i},
\]
we can describe the above equality as
\[
\dim \Ker (\widehat{A} - \widehat{\alpha\,\unit_V}) = 
\sum_{\beta,j} \delta(\alpha,\beta;j) \dim V^j(\beta).
\]
Next we rewrite \eqref{eq:centralizer} as
\begin{align*}
\dim Z(A) 
&= \sum_\alpha \sum_j \bigl( \dim V^j(\alpha) \bigr)^2 
+\sum_{i=2}^k \sum_{(\alpha_i,\dots ,\alpha_k)} 
\Bigl( \sum_\beta \sum_j 
\delta_{\alpha_i,\beta_i} \cdots \delta_{\alpha_k,\beta_k}
\dim V^j(\beta) \Bigr)^2 \\
&= \sum_\alpha \sum_j \bigl( \dim V^j(\alpha) \bigr)^2 \\
&\qquad +\sum_{i=2}^k \sum_{(\alpha_i,\dots ,\alpha_k)} 
\sum_{\beta,\gamma} \sum_{j,l} 
\delta_{\alpha_i,\beta_i} \cdots \delta_{\alpha_k,\beta_k}
\delta_{\alpha_i,\gamma_i} \cdots \delta_{\alpha_k,\gamma_k}
\dim V^j(\beta) \dim V^l(\gamma)  \\
&= \sum_\alpha \sum_j \bigl( \dim V^j(\alpha) \bigr)^2 
+\sum_{i=2}^k 
\sum_{\beta,\gamma} \sum_{j,l} 
\delta_{\beta_i,\gamma_i} \cdots \delta_{\beta_k,\gamma_k}
\dim V^j(\beta) \dim V^l(\gamma).
\end{align*}
Using the inequality $\dim V^j(\alpha) \leq \dim V(\alpha)$, we obtain
\begin{align*}
\dim Z(A) &\leq \sum_\alpha \sum_j \dim V^j(\alpha) \dim V(\alpha) \\
&\qquad +\sum_{i=2}^k 
\sum_{\beta,\gamma} \sum_{j,l} 
\delta_{\beta_i,\gamma_i} \cdots \delta_{\beta_k,\gamma_k}
\dim V^j(\beta) \dim V^l(\gamma) \\
&= \sum_{\beta,\gamma} \sum_{j,l} 
\bigl( \delta_{l,1} \delta_{\beta,\gamma} 
+ \sum_{i=2}^k \delta_{\beta_i,\gamma_i} \cdots \delta_{\beta_k,\gamma_k} \bigr)
\dim V^j(\beta) \dim V^l(\gamma) \\
&= \sum_{\beta,\gamma} \sum_{j,l} 
\delta(\beta,\gamma;l) \dim V^j(\beta) \dim V^l(\gamma) \\
&= \sum_{\beta,j} \dim \Ker (\widehat{A} - \widehat{\beta\,\unit_V}) \dim V^j(\beta).
\end{align*}
Take $\alpha \in \EE_k(\C)$ to attain the maximum of the values 
$\dim \Ker (\widehat{A} - \widehat{\alpha\,\unit_V})$. 
Then
\[
\dim Z(A) \leq \sum_{\beta,j} 
\dim \Ker (\widehat{A} - \widehat{\alpha\,\unit_V}) \dim V^j(\beta)
=\dim \Ker (\widehat{A} - \widehat{\alpha\,\unit_V}) \dim V,
\]
which is the desired inequality.
\end{proof}

\begin{lemma}\label{lem:Scott}
Suppose that a pair $(V,A) \in \DDk^0$ is irreducible 
and $(V,A) \not\sim (\C,0)$.
Then the rank of $\HD(V,A)$ is greater than or equal to $2 \dim V$.
\end{lemma}

\begin{proof}
By the assumption and Lemma~\ref{lem:irred},
$(V,W,0,T,Q,P):=\kappa(V,A)$ is irreducible.
Since the pairs $(\range Q, W)$ and $(\Ker P,0)$ are both subrepresentations 
of it, we see that $Q$ is surjective and $P$ is injective
if $W \neq 0$, in other words, if $(V,A) \not\sim (\C,0)$ 
(see Corollary~\ref{cor:exceptional}).
On the other hand, we have $QP =-\Res_{z=\infty} A(z) =0$. 
Thus under the assumption $(V,A) \not\sim (\C,0)$, we obtain 
\[
\dim V = \rank P \leq \dim \Ker Q = \dim W - \dim V. 
\]
\end{proof}

\begin{theorem}\label{thm:Katz}
Suppose that a pair $(V,A) \in \DDk^0$ with $\dim V \geq 2$ 
is irreducible, naively rigid, 
and has a normal form at any $t \in D$.
Then there exists a rank 1 system $\alpha(z) \in \EE_\vek(\C)$ 
such that the rank of  
$\mc_{\lambda/\zeta}\circ \add_{-\alpha} (V,A)$, 
where $\lambda = -\Res_{z=\infty} \alpha(z)$, 
is less than $\dim V$.
\end{theorem}

\begin{proof}
For each $t \in D$, take $\alpha_t(z) \in \EE_{k_t}(\C)$ 
to satisfy the condition in Lemma~\ref{lem:Katz} for $A_t(z)$, 
and set $\alpha(z):= \sum_{t \in D} \alpha_t(z-t)$. 
Let $W$ be the underlying vector space of $\HD \circ \add_{-\alpha} (V,A)$.
Then, since $(V,A)$ is naively rigid, 
using the definition of $W$ we have
\begin{align*}
2 \dim \GL(V) -2 &= \dim \bO(A) \\
&= \sum_{t \in D} \bigl( \dim G_{k_t}(V) - \dim Z(A_t) \bigr) \\
&= \dim \GL(V) \sum_{t \in D} k_t - \sum_{t \in D} \dim Z(A_t) \\
&\geq \dim V \sum_{t \in D} \Bigl( k_t \dim V - 
 \dim \Ker (\widehat{A}_t - \widehat{\alpha_t\,\unit_V}) \Bigr) \\
&= \dim V \dim W.
\end{align*}
Hence 
\[
\dim W \leq (2 \dim \GL(V) -2)/ \dim V < 2 \dim V.
\]
Note that the pair $(V,A-\alpha\,\unit_V)$ is irreducible as so is $(V,A)$.
Therefore the above inequality together with Lemma~\ref{lem:Scott} implies 
$\lambda = -\Res_{z=\infty} \alpha(z) \neq 0$.
By Example~\ref{ex:ref}, we thus have 
\[
\dim V^{\lambda/\zeta} = \dim W - \dim V < \dim V.
\]
\end{proof}

Note that the functor 
$\add_\alpha \circ \mc_{\lambda/\zeta} \circ \add_\alpha$, 
where $\lambda=\Res_{z=\infty} \alpha(z)$, 
preserves the full subcategory $\DDk^0$ 
for any $\alpha(z) \in \EE_\vek(\C)$ by Example~\ref{ex:ref}.
Therefore the above theorem together with Proposition~\ref{prop:mc-sing} 
implies the following:
\begin{corollary}\label{cor:Katz}
Suppose that a pair $(V,A) \in \DDk^0$ with $\dim V \geq 2$ 
is irreducible, naively rigid, 
and has a normal form at any $t \in D$.
Then applying a suitable finite iteration of operations of the form
\begin{equation}\label{eq:mc-modified}
\add_\alpha \circ \mc_{\lambda/\zeta} \circ \add_\alpha, \quad
\alpha \in \EE_\vek(\C),\; \lambda =\Res_{z=\infty} \alpha(z),
\end{equation}
makes $(V,A)$ into an irreducible pair of rank 1.
\end{corollary}

\begin{example}\label{ex:rigid}
(a) Clearly all rank 1 objects in $\DDk^0$ are irreducible and naively rigid.

(b) Let $(\C^2,A) \in \DDk^0$ be an irreducible rank 2 object 
having a normal form 
\[
\Lambda_t (z) = 
\left(\, 
\begin{matrix} \alpha_t(z) & 0 \\ 0 & \beta_t(z) \end{matrix}
\,\right), \quad \alpha_t(z),\, \beta_t(z) \in \EE_{k_t}(\C),
\]
at each $t \in D$. 
The addition functor with $-\beta(z) := -\sum_{t \in D} \beta_t(z-t)$
sends $(\C^2,A)$ to $(\C^2,A-\beta\,\unit_{\C^2})$, which has the normal form 
$\Lambda_t-\beta_t\,\unit_{\C^2}=(\alpha_t -\beta_t) \oplus 0$ at each $t$.
Let $(\C^2,W_t,N_t,Q_t,P_t)$ be 
the canonical datum for $\Lambda_t -\beta_t\,\unit_{\C^2}$.
As in the proof of Lemma~\ref{lem:mc-sing}, we then see that
$W_t \simeq \C \otimes \C^{d_t}=\C^{d_t}$, 
where $d_t=\ord(\alpha_t -\beta_t)$, 
and $N_t$ is given by the nilpotent single Jordan block.
Thus we have $G_{N_t} \simeq G_{d_t}(\C) \subset \GL(\C^{d_t})$, which is abelian.
In particular all elements in $\g_{N_t}^*$ are fixed by the coadjoint action.
Therefore in Theorem~\ref{thm:geom}, 
the $G_\vek(V)$-coadjoint orbit $\bO(A)-\beta\,\unit_{\C^2}$
is described as 
$\mu_T^{-1}(\OO)^{T\st}/G_T$ with $\OO$ being just a single (central) element.
Thus we can compute the dimension of the naive moduli space as
\begin{align*}
\dim \Mreg(\bO(A),0)
&=\dim \Mreg \bigl( \bO(A)-\beta\,\unit_{\C^2},
-w^{-1}\Res_{z=\infty}\beta(z)\,\unit_{\C^2} \bigr) \\
&= 2 \sum_{t \in D} \dim \bM(\C^2,W_t) 
 - 2 \sum_{t \in D} \dim G_{N_t} -2 \dim \GL(\C^2) +2 \\
&= 4 \sum_{t \in D} d_t - 2\sum_{t \in D} d_t -6 \\
&= 2 \sum_{t \in D} d_t -6.
\end{align*}
Now suppose that $(\C^2,A)$ is naively rigid. 
Then the above formula implies that
the number of points $t \in D$ with $d_t >0$ is at most 3, 
so we may assume $D=\{ t_1, t_2, t_3 \}$, 
and up to permutation the triple $(d_{t_1},d_{t_2},d_{t_3})$ 
is one of the following:
\[
(3,0,0), \quad (2,1,0), \quad (1,1,1).
\]
In all the cases, the operator \eqref{eq:mc-modified} with $\alpha:=-\beta$
makes $(\C^2,A)$ into a rank 1 object.
\end{example}

\begin{remark}\label{rem:painleve}
In the situation of Example~\ref{ex:rigid}, (b), 
the case $\dim \Mreg(\bO(A),0)=2$ is also important.
In this case, the number of points $t \in D$ with $d_t>0$ is at most 4,
so we may assume $D=\{ t_1, t_2, t_3, t_4 \}$, 
and up to permutation the quadruple $(d_{t_1},d_{t_2},d_{t_3},d_{t_4})$ 
is one of the following:
\[
(4,0,0,0), \quad (3,1,0,0), \quad (2,2,0,0), \quad (2,1,1,0), \quad (1,1,1,1).
\]
Unlike in the case $\dim \Mreg(\bO(A),0)=0$,
the rank of $(\C^2,A)$ {\em does not change} under 
the operation \eqref{eq:mc-modified} with $\alpha:=-\beta$.
Indeed, Example~\ref{ex:ref} implies that the rank of the resulting system is 
equal to $\sum_t \dim W_t - \dim \C^2 = 4 - 2 =2$.
It is well-known~\cite{Okamoto} that through the isomonodromic deformation,
the above quadruples correspond to 
the Painlev\'e equations of type II, IV, III, V, VI respectively.
\end{remark}

\begin{acknowledgements}\small
The author is extremely grateful to 
Professor Yoshishige Haraoka, Toshio Oshima and Masa-Hiko Saito for 
giving him many opportunities to discuss problems around 
the middle convolution. 
Also, the author is much obliged to 
Philip Boalch for valuable comments, 
and to Hiroshi Kawakami and Professor Kouichi Takemura 
for expounding their works to him. 
Kawakami's talk triggered the author's interest 
in the problems treated in this article. 
Takemura suggested in his talk 
a direct generalization of Dettweiler-Reiter's description 
of the middle convolution for irregular singular systems 
(in one parameter case, however systems are allowed to have 
a pole of arbitrary order at $\infty$),
and conjectured that it holds basic properties of the original middle convolution.
His idea is used on \eqref{eq:matrix}.

Finally, the author would like to thank the referee, 
who ponited out errors and a lack of argument 
in the earlier version of this article.
\end{acknowledgements}


\begin{thebibliography}{99}\small
\setlength{\itemsep}{.1em}
\bibitem{AHH-dual}
Adams, M. R., Harnad, J., Hurtubise, J.: 
Dual moment maps into loop algebras. 
Lett.\ Math.\ Phys.\ \textbf{20}, 299--308 (1990)

\bibitem{AHH-iso2}
Adams, M. R., Harnad, J., Hurtubise, J.:
Isospectral Hamiltonian flows 
in finite and infinite dimensions. II.
Integration of flows.
Comm.\ Math.\ Phys.\ \textbf{134}, 555--585 (1990)

\bibitem{AHP-iso1}
Adams, M. R., Harnad, J., Previato, E.: 
Isospectral Hamiltonian flows 
in finite and infinite dimensions. I. 
Generalized Moser systems and moment maps into loop algebras. 
Comm.\ Math.\ Phys.\ \textbf{117}, 451--500 (1988)

\bibitem{Arin-mc}
Arinkin, D.:
Fourier transform and middle convolution for 
irregular $\mathscr{D}$-modules.
arXiv.0808.0699

\bibitem{Arin-rigid}
Arinkin, D.:
Rigid irregular connections on $\CP^1$.
arXiv:0808.0742

\bibitem{BV-group}
Babbitt, D. G., Varadarajan, V. S.:
Formal reduction theory of meromorphic differential equations: 
a group theoretic view. 
Pacific J.\ Math.\ \textbf{109}, 1--80 (1983) 

\bibitem{BV-ast}
Babbitt, D. G., Varadarajan, V. S.:
Local moduli for meromorphic differential equations.
Ast\'erisque \textbf{169--170}, 1--217 (1989)

\bibitem{BJL1}
Balser, W., Jurkat, W. B., Lutz, D. A.:
On the reduction of connection problems for differential
equations with an irregular singular point to ones with only
regular singularities. I.
SIAM J.\ Math.\ Anal.\ \textbf{12}, 691--721 (1981)


\bibitem{BE}
Bloch, S., Esnault, H.:
Local Fourier transforms and rigidity for $\mathscr{D}$-modules.
Asian J.\ Math.\ \textbf{8}, 587--605 (2004)

\bibitem{Boa-thesis}
Boalch, P.:
Symplectic manifolds and isomonodromic deformations.
Adv.\ Math.\ \textbf{163}, 137--205 (2001)

\bibitem{Boa-klein}
Boalch, P.:
From Klein to Painlev\'e via Fourier, Laplace and Jimbo.
Proc.\ London Math.\ Soc.\ (3) \textbf{90}, 167--208 (2005)

\bibitem{Boa-diff}
Boalch, P.:
Quivers and difference Painlev\'e equations.
In: Groups and symmetries, pp.~25--51.
CRM Proc. Lecture Notes, vol.~47. 
Amer.\ Math.\ Soc., Providence (2009)

\bibitem{Boa-quiver}
Boalch, P.:
Irregular connections and Kac-Moody root systems.
arXiv:0806.1050

\bibitem{Cra-geom}
Crawley-Boevey, W.:
Geometry of the moment map for representations of quivers.
Compositio Math.\ \textbf{126}, 257--293 (2001)

\bibitem{Cra-par}
Crawley-Boevey, W.: 
Indecomposable parabolic bundles and 
the existence of matrices in prescribed conjugacy class closures 
with product equal to the identity. 
Publ.\ Math.\ Inst.\ Hautes \'Etudes Sci.\ No.\ 100, 171--207 (2004)

\bibitem{CS}
Crawley-Boevey, W., Shaw, P.:
Multiplicative preprojective algebras, 
middle convolution and the Deligne-Simpson problem.
Adv.\ Math.\ \textbf{201}, 180--208 (2006)

\bibitem{DR-mc}
Dettweiler, M., Reiter, S.:
An algorithm of Katz and its application to the inverse Galois problem.
J.\ Symbolic Comput.\ \textbf{30}, 761--798 (2000)

\bibitem{DR-rh}
Dettweiler, M., Reiter, S.:
Middle convolution of Fuchsian systems and the construction 
of rigid differential systems.
J.\ Algebra \textbf{318}, 1--24 (2007)

\bibitem{DR-p}
Dettweiler, M., Reiter, S.:
Painlev\'e equations and the middle convolution.
Adv.\ Geom.\ \textbf{7}, 317--330 (2007)

\bibitem{Filipuk}
Filipuk, G.:
On the middle convolution and 
birational symmetries of the sixth Painlev\'e equation.
Kumamoto J.\ Math.\ \textbf{19}, 15--23 (2006)


\bibitem{Harnad}
Harnad, J.:
Dual isomonodromic deformations and 
moment maps to loop algebras.
Comm.\ Math.\ Phys.\ \textbf{166}, 337--365 (1994)

\bibitem{HF}
Haraoka, Y., Filipuk, G.:
Middle convolution and deformation for Fuchsian systems.
J.\ Lond.\ Math.\ Soc.\ (2) \textbf{76}, 438--450 (2007)

\bibitem{HY}
Haraoka, Y., Yokoyama, T.: 
Construction of rigid local systems 
and integral representations of their sections. 
Math.\ Nachr.\ \textbf{279} 255--271 (2006)

\bibitem{Huk}
Hukuhara, M.:
Sur les points singuliers des \'equations diff\'erentielles lin\'eaires, II. 
J.\ Fac.\ Sci.\ Hokkaido Univ.\ \textbf{5}, 123--166 (1937)


\bibitem{JMMS}
Jimbo, M., Miwa, T., M\^ori, Y., Sato, M.:
Density matrix of an impenetrable Bose gas 
and the fifth Painlev\'e transcendent.
Phys.\ D \textbf{1}, 80--158 (1980)

\bibitem{Katz}
Katz, N. M.:
Rigid local systems.
Annals of Mathematics Studies, vol.~139. 
Princeton University Press, Princeton (1996)

\bibitem{Kawakami}
Kawakami, H.:
Generalized Okubo systems and the middle convolution.
Thesis, The University of Tokyo (2009)

\bibitem{Kostov}
Kostov, V.:
The Deligne-Simpson problem---a survey.
J.\ Algebra \textbf{281}, 83--108 (2004)

\bibitem{King}
King, A. D.:
Moduli of representations of finite-dimensional algebras.
Quart.\ J.\ Math.\ Oxford Ser.\ (2) \textbf{45}, 515--530 (1994)

\bibitem{LP}
Le~Bruyn, L., Procesi, C.:
Semisimple representations of quivers.
Trans.\ Amer.\ Math.\ Soc.\ \textbf{317}, 585--598 (1990)

\bibitem{Levelt}
Levelt, A. H. M.: 
Jordan decomposition for a class of singular differential operators. 
Ark.\ Mat.\ \textbf{13}, 1--27 (1975)

\bibitem{Lus-qvar}
Lusztig, G.:
On quiver varieties. 
Adv.\ Math.\ \textbf{136}, 141--182 (1998)


\bibitem{GIT}
Mumford, D., Fogarty, J., Kirwan, F.:
Geometric invariant theory, 3rd edn. 
Ergebnisse der Mathematik und ihrer Grenzgebiete (2)34. 
Springer, Berlin (1994)

\bibitem{Nak-duke1}
Nakajima, H.:
Instantons on ALE spaces, quiver varieties, and Kac-Moody algebras. 
Duke Math.\ J.\ \textbf{76}, 365--416 (1994)



\bibitem{Okamoto}
Okamoto, K.: 
Isomonodromic deformation and Painlev\'e equations, and the Garnier system. 
J.\ Fac.\ Sci.\ Univ.\ Tokyo Sect.\ IA Math.\  \textbf{33}, 575--618 (1986)

\bibitem{Okubo}
Okubo, K.: 
Connection problems for systems of linear differential equations.
In: Urabe, M. (ed.) 
Japan-United States Seminar on Ordinary Differential and Functional Equations (Kyoto, 1971), Lecture Notes in Math., Vol.\ 243, 
pp.\ 238--248. Springer, Berlin (1971)

\bibitem{Oshima-conj}
Oshima, T.:
A quantization of conjugacy classes of matrices.
Adv.\ Math.\ \textbf{196}, 124--146 (2005)

\bibitem{Oshima-conn}
Oshima, T.:
Classification of Fuchsian systems 
and their connection problem.
arXiv:0811.2916


\bibitem{Tur}
Turrittin, H. L.: 
Convergent solutions of ordinary linear homogeneous differential equations 
in the neighborhood of an irregular singular point. 
Acta Math.\ \textbf{93}, 27--66 (1955)

\bibitem{Volk}
V\"olklein, H.:
The braid group and linear rigidity.
Geom.\ Dedicata \textbf{84}, 135--150 (2001)

\bibitem{Wein}
Weinstein, A.:
The local structure of Poisson manifolds.
J.\ Differential Geom.\ \textbf{18}, 523--557 (1983)

\bibitem{Wood}
Woodhouse, N. M. J.:
Duality for the general isomonodromy problem.
J.\ Geom.\ Phys.\ \textbf{57}, 1147--1170 (2007)


\bibitem{Yok2}
Yokoyama, T.: 
Construction of systems of differential equations 
of Okubo normal form with rigid monodromy. 
Math.\ Nachr.\ \textbf{279}, 327--348 (2006)

\end{thebibliography}
\end{document}